\documentclass[11pt, reqno]{article}
\usepackage[T1]{fontenc}
\usepackage{amsmath,amssymb,amsfonts,amsthm}
\usepackage{color}

\usepackage[colorlinks=true,linkcolor=blue,citecolor=blue,urlcolor=blue,pdfborder={0 0 0}]{hyperref}
\usepackage{cleveref}

\usepackage{caption}
\usepackage{thmtools}

\usepackage{mathrsfs}

\crefname{theorem}{Theorem}{Theorems}
\crefname{thm}{Theorem}{Theorems}
\crefname{lemma}{Lemma}{Lemmas}
\crefname{lem}{Lemma}{Lemmas}
\crefname{remark}{Remark}{Remarks}
\crefname{prop}{Proposition}{Propositions}
\crefname{defn}{Definition}{Definitions}
\crefname{corollary}{Corollary}{Corollaries}
\crefname{conjecture}{Conjecture}{Conjectures}
\crefname{question}{Question}{Questions}
\crefname{chapter}{Chapter}{Chapters}
\crefname{section}{Section}{Sections}
\crefname{figure}{Figure}{Figures}
\crefname{example}{Example}{Examples}

\theoremstyle{plain}
\newtheorem{thm}{Theorem}[section]

\newtheorem{lemma}[thm]{Lemma}
\newtheorem{theorem}[thm]{Theorem}

\newtheorem{lem}[thm]{Lemma}

\newtheorem{corollary}[thm]{Corollary}

\newtheorem{prop}[thm]{Proposition}
\newtheorem{conjecture}[thm]{Conjecture}

\theoremstyle{definition}
\newtheorem{defn}[thm]{Definition}

\theoremstyle{remark}
\newtheorem{remark}[thm]{Remark}

\numberwithin{equation}{section}

\renewcommand{\P}{\mathbb P}
\newcommand{\E}{\mathbb E}

\newcommand{\R}{\mathbb R}
\newcommand{\Z}{\mathbb Z}

\usepackage[margin=1in]{geometry}
\usepackage{tikz}
\usepackage{pgfplots}
\usepackage{extarrows}
\usepackage{titling}
\usepackage{commath}
\usepackage{wasysym}
\usepackage{mathtools}
\usepackage{titlesec}

\usepackage{bbm}
\usepackage{setspace}
\usepackage{enumitem}
\usepackage{booktabs}

\usepackage[textsize=tiny]{todonotes}

\usepackage{cite}

\newcommand{\Aut}{\operatorname{Aut}}

\def\P{\mathbb{P}}

\DeclareMathSymbol{\leqslant}{\mathalpha}{AMSa}{"36} 
\DeclareMathSymbol{\geqslant}{\mathalpha}{AMSa}{"3E} 
\DeclareMathSymbol{\eset}{\mathalpha}{AMSb}{"3F}     

\renewcommand{\epsilon}{\varepsilon}

\makeatletter
\pgfdeclareshape{slash underlined}
{
  \inheritsavedanchors[from=rectangle] 
  \inheritanchorborder[from=rectangle]
  \inheritanchor[from=rectangle]{north}
  \inheritanchor[from=rectangle]{north west}
  \inheritanchor[from=rectangle]{north east}
  \inheritanchor[from=rectangle]{center}
  \inheritanchor[from=rectangle]{west}
  \inheritanchor[from=rectangle]{east}
  \inheritanchor[from=rectangle]{mid}
  \inheritanchor[from=rectangle]{mid west}
  \inheritanchor[from=rectangle]{mid east}
  \inheritanchor[from=rectangle]{base}
  \inheritanchor[from=rectangle]{base west}
  \inheritanchor[from=rectangle]{base east}
  \inheritanchor[from=rectangle]{south}
  \inheritanchor[from=rectangle]{south west}
  \inheritanchor[from=rectangle]{south east}
  \inheritanchorborder[from=rectangle]
  \foregroundpath{
    \southwest \pgf@xa=\pgf@x \pgf@ya=\pgf@y
    \northeast \pgf@xb=\pgf@x \pgf@yb=\pgf@y
    \pgf@xc=\pgf@xa
    \advance\pgf@xc by .5\pgf@xb
    \pgf@yc=\pgf@ya
    \advance\pgf@xc by -1.3pt
    \advance\pgf@yc by -1.8pt
    \pgfpathmoveto{\pgfqpoint{\pgf@xc}{\pgf@yc}}
    \advance\pgf@xc by  2.6pt
    \advance\pgf@yc by  3.6pt
    \pgfpathlineto{\pgfqpoint{\pgf@xc}{\pgf@yc}}
    \pgfpathmoveto{\pgfqpoint{\pgf@xa}{\pgf@ya}}
    \pgfpathlineto{\pgfqpoint{\pgf@xb}{\pgf@ya}}
 }
}
\tikzset{nomorepostaction/.code=\let\tikz@postactions\pgfutil@empty}

\makeatother

\newcommand{\MP}{\color{blue}}
\newcommand{\eMP}{\normalcolor}

\pretitle{\begin{flushleft}\Large}
\posttitle{\par\end{flushleft}\vskip 0.5em
}
\preauthor{\begin{flushleft}}
\postauthor{
\\
\vspace{0.5em}
\footnotesize{The Division of Physics, Mathematics and Astronomy, California Institute of Technology\\ Email: 
\href{mailto:t.hutchcroft@caltech.edu}{t.hutchcroft@caltech.edu}, \href{mailto:mpan2@caltech.edu}{mpan2@caltech.edu}}
\end{flushleft}}
\predate{\begin{flushleft}}
\postdate{\par\end{flushleft}}
\title{{\bf Percolation at the uniqueness threshold via subgroup relativization}}

\renewenvironment{abstract}
 {\par\noindent\textbf{\abstractname.}\ \ignorespaces}
 {\par\medskip}

\titleformat{\section}
  {\normalfont\large \bf}{\thesection}{0.4em}{}

\author{{\bf Tom Hutchcroft and Minghao Pan}}

\begin{document}

\maketitle

\begin{abstract}
 We study percolation on nonamenable groups at the \emph{uniqueness threshold} $p_u$, the critical value that separates the phase in which there are infinitely many infinite clusters from the phase in which there is a unique infinite cluster.
   The number of infinite clusters at $p_u$ itself is a subtle question, depending on the choice of group, with only a relatively small number of examples understood. In this paper, we do the following:
   \begin{itemize}
   \item Prove non-uniqueness at $p_u$ in a new class of examples, namely those groups that contain an amenable, $wq$-normal subgroup of exponential growth. Concrete new examples to which this result applies include lamplighters over nonamenable base groups.
   \item Prove a co-heredity property of a certain strong form of non-uniqueness at $p_u$, stating that this property is inherited from a $wq$-normal subgroup to the entire group.
    Remarkably, this co-heredity property is the same as that proven for the vanishing of the first $\ell^2$ Betti number by Peterson and Thom (Invent.\ Math.\ 2011), supporting the conjecture that the two properties are equivalent.
   \end{itemize}
   Our proof is based on the method of \emph{subgroup relativization}, and relies in particular on relativized versions of uniqueness monotonicity, the equivalence of non-uniqueness and connectivity decay, 
    the sharpness of the phase transition, and the Burton-Keane theorem.
     As a further application of the relative Burton-Keane theorem, we resolve a question of Lyons and Schramm (Ann.\ Probab.\ 1999) concerning intersections of random walks with percolation clusters.
\end{abstract}

\setstretch{1.1}


\section{Introduction}

Let $G$ be a connected graph that is locally finite, meaning that all vertices have finite degree. \textbf{Bernoulli bond percolation} on $G$, which we denote by $G_{p}$, is obtained from $G$ by independently 
setting each edge \textbf{open} (retained) or \textbf{closed} (deleted), with probability $p$ of being open. Connected components of the resulting random subgraph are called \textbf{clusters}. 
We denote the law of $G_{p}$ by $\mathbb{P}_{p}$ and the
corresponding expectation by $\mathbb{E}_{p}$. We will be interested primarily in the case that $G$ is transitive, meaning that any vertex can be mapped to any other vertex by a symmetry of the graph, and will restrict to this case in the following discussion unless otherwise indicated.

A central problem of percolation theory is to understand the \emph{number of infinite clusters}, denoted $N_\infty$, and how this number depends on the parameter $p$.
To this end, we define the \textbf{critical
parameter}  to be%
\begin{equation*}
p_{c}=p_{c}(G):=\inf \left\{ p\in \lbrack 0,1]:
\text{$N_\infty>0$ $\mathbb{P}_{p}$ a.s.}\right\} 
\end{equation*}%
and define the \textbf{uniqueness threshold} to be%
\begin{equation*}
p_{u}=p_{u}(G)=\inf \left\{ p\in \lbrack 0,1]:
\text{$N_\infty=1$ 
$\mathbb{P}_{p}$ a.s.}\right\};
\end{equation*}
The fact that these two critical values demarcate the \emph{only} two changes in the value of $N_\infty$ follow from results of
Newman and Schulman \cite{NewSch81}, who proved that $N_\infty$ is non-random and must belong to $\{0,1,\infty\}$ a.s., and H\"{a}ggstr\"{o}m, Peres, and Schonmann \cite{HPS99} who proved that $N_\infty=1$ a.s.\ for every $p>p_u$.
Thus, there are no infinite clusters when $p\in [0,p_c)$, infinitely many infinite clusters when $p\in (p_c,p_u)$, and a single infinite cluster when $p\in (p_u,1]$. Note however that the second two phases are \emph{not} guaranteed to be non-empty: 
It is a classical theorem, essentially due to Burton and Keane~\cite{burton1989density} (see also~\cite{gandolfi1992uniqueness}), that $p_{c}(G)=p_{u}(G)$ for amenable transitive graphs, so that the middle phase $(p_c,p_u)$ is empty for such graphs. The converse to Burton-Keane, stating that this middle phase is non-empty on every nonamenable transitive graph, is a famous conjecture of Benjamini and Schramm
  \cite{BS96a}; this conjecture has been proven for several classes of transitive graphs
   including ``highly nonamenable graphs'' \cite{bperc96,MR1756965,MR1833805,MR3005730}, Gromov-hyperbolic graphs \cite{1804.10191}, groups of cost $>1$ \cite{MR3009109}, and graphs admitting a nonunimodular transitive subgroup of automorphisms \cite{Hutchcroftnonunimodularperc} but remains open in general. 
   (It is known that every nonamenable \emph{group} has at least one Cayley graph for which the conjecture is true \cite{MR1756965}.)
   Finally, the uniqueness phase $(p_u,1)$ of a transitive graph is conjectured to be non-trivial if and only if the graph is one-ended \cite{bperc96}, and this is known to be true for finitely presented groups \cite{MR1622785,MR2286517}, wreath products of infinite groups with finite groups \cite{LS99}, Kazhdan groups \cite{LS99}, nonamenable products \cite{MR1770624}, amenable graphs (where it is equivalent to $p_c<1$) \cite{MR4181032}, and graphs of internal isoperimetric dimension at least twelve \cite{hutchcroft2021non}; the converse fact that groups with more than one end have $p_u=1$ is trivial.

The number of infinite clusters at the critical thresholds $p_c$ and $p_u$ themselves appears to be an extremely subtle problem.
At the \emph{existence threshold} $p_c$, it is conjectured \cite{bperc96} that there are no infinite clusters at $p_c$ whenever $p_c<1$ or, equivalently, whenever the graph is not one-dimensional~\cite{MR4181032}. This conjecture has received a huge amount of attention over many years, particularly in the classical case of Euclidean lattices where it is solved in two dimensions \cite{harris1960lower,kesten1980critical,MR3503025} and high dimensions \cite{MR1043524,fitzner2015nearest} but remains a notorious open problem for dimensions $3\leq d \leq 6$. Outside the Euclidean setting, the conjecture was solved for nonamenable Cayley graphs in~\cite{BLPS99b}, for transitive graphs of exponential growth in \cite{Hutchcroft2016944,timar2006percolation}, and for certain transitive graphs of intermediate volume growth in \cite{hermon2021no}. Although the qualitative problem about the non-existence of infinite clusters at criticality for percolation on transitive nonamenable graphs is solved, it remains an important open problem to establish a sharp \emph{quantitative} understanding of critical percolation in the nonamenable setting; see \cite{1808.08940,1804.10191} for partial results and further discussion.

In this paper, we instead focus on understanding percolation at the \emph{uniqueness threshold}. In contrast to the existence threshold, there are no general conjectures for whether or not there is uniqueness at $p_u$, with the answer known to depend on the choice of graph in a way that is not understood even at the heuristic level. Indeed, it is not even known whether the number of infinite clusters at $p_u$ in a Cayley graph depends on the choice of the generators for the group, and surprising recent results of Pete and Tim\'ar for a different model \cite{pete2022free} suggest that it might. Moreover, the problem is understood in only a relatively small number of cases: Infinitely-ended graphs and nonamenable planar graphs have uniqueness at $p_u$ \cite{BS00}, while nonamenable products \cite{MR1770624,gaboriau2016approximations} (see also \cite{schonmann1999percolation}) and Cayley graphs of Kazhdan groups \cite{LS99} have non-uniqueness at $p_u$.  Intriguingly, in all examples that have been analyzed, uniqueness at $p_u$ coincides with positivity of the first $\ell^2$-Betti number (see \cite[Chapter 10.8]{LP:book} for background), but no implication linking this property with uniqueness at $p_u$ has been proven in either direction. In summary, progress on the problem has been slow, with no clear heuristics emerging and essentially no new examples analyzed in the last twenty years.



The main contribution of this paper is to develop a new perspective on the problem, which we call \emph{subgroup relativization}, that can be used to prove non-uniqueness at $p_u$ in a large number of new examples. Before describing this method and its most general applications, let us first state a concrete theorem encapsulating one of its most striking applications.


\begin{theorem}\label{thm:lamp}
Let $\Gamma$ be a finitely generated nonamenable group and consider the lamplighter group $\Z_2 \wr \Gamma$. Every Cayley graph
 of this lamplighter group has non-uniqueness at $p_u$.
\end{theorem}

The tools we develop in this paper are also applied in the companion paper \cite{HutchcroftPan2} to establish a \emph{quantitative} understanding of the uniqueness transition for certain graphs including the product of a tree and a line. In particular, we prove that these graphs have $p_{2\to2}=p_u$, where $p_{2\to2}$ is the \emph{$L^2$ boundedness threshold} introduced in \cite{1804.10191}; these are the first examples proven to have this property.

\subsection{Subgroup relativization}

At a high level, our new framework of subgroup relativization can be described as follows. Suppose that we have a Cayley graph $G$ of a finitely generated group $\Gamma$, and wish to understand the behaviour of percolation on $G$. Rather than directly studying the partition of $\Gamma$ into clusters of $G_p$, we can instead study the induced partition on a \emph{subgroup} $H$ of $\Gamma$, i.e., the random partition of $H$ defined by $\{ K \cap H : K$ a cluster of $G_p\}$. 
As we shall see, this framework is useful in part because several of the most important and classical theorems in percolation theory, including the \emph{Burton-Keane theorem}, the \emph{sharpness of the phase transition}, \emph{uniqueness monotonicity}, and the \emph{equivalence of non-uniqueness and connectivity decay} admit \emph{relative} versions applying to the study of these random subgroup partitions. Building on this, we show that several results concerning percolation at $p_c$ and $p_u$ established for various classes of groups in e.g.\cite{LS99,MR1770624,MR1676831,Hutchcroft2016944} also admit ``relativized'' generalizations, where the properties required by the hypotheses of these theorems need to hold only for the \emph{subgroup} rather than the whole group. This line of thought eventually leads to the definition of \emph{obstacles to uniqueness} and their co-heredity properties, 
 which underlie the most general forms of our results.

\medskip

In this subsection we state the relativized versions of these classical theorems,
 with applications to the problem of non-uniqueness at $p_u$ given in the following subsection. 
We begin by defining the relative versions of $p_c$ and $p_u$.

\begin{defn}
Let $G$ be a connected, locally finite graph, and let $A$ be a set of vertices in $G$. We say that a percolation cluster $K$ is \textbf{$A$-infinite} if $|K\cap A|=\infty$. 
We define the \textbf{relative critical probability} to be
\[
p_c(A;G) := \inf\{p: \text{there exists an $A$-infinite cluster in } G_p \text{ a.s.}\} 
\]
and \textbf{relative uniqueness threshold} $p_u(A;G)$ to be
\[
p_u(A;G) := \inf\{p: \text{there exists a unique $A$-infinite cluster in } G_p \text{ a.s.}\}.
\]
Note that the existence of an $A$-infinite cluster is a tail event, so that it has probability either zero or one. We will be particularly interested in $p_c(H;G)$ and $p_u(H;G)$ when $G$ is a Cayley graph of a group $\Gamma$ and $H$ is an infinite subgroup of $\Gamma$. In this case, $H$ acts on percolation configurations by $L_h(\omega(x,y))=\omega(hx,hy)$, where $h\in H$ and $(x,y)$ is an edge in $G$ (a similar action of $H$ exists when $G$ is not simple), and the law of Bernoulli percolation on $G$ is $H$-ergodic, meaning that any $H$-invariant random variable is a constant a.s. This leads to the following relativized version of the  Newman-Schulman theorem \cite{NewSch81}:
\end{defn}

\begin{lemma}[Relative Newman-Schulman]
\label{lem:Newman_Schulman}
Let $G$ be a Cayley graph of a group $\Gamma$ and $H$ be an infinite subgroup of $\Gamma$. For each $p\in [0,1]$ there exists $k_p\in \{0,1,\infty\}$ such that the number of $H$-infinite clusters in $G_p$ is equal to $k_p$ a.s.
\end{lemma}

The proof of this lemma is simple enough to be given immediately:

\begin{proof}[Proof of \cref{lem:Newman_Schulman}]
By $H$-ergodicity, the number of $H$-infinite percolation clusters in $G_p$ is equal to some constant $k_p$ a.s. Suppose for contradiction that $2\leq k_p \leq \infty$. By continuity of measure, there exists $r\geq 1$ such that the ball $B_r(o)$ of radius $r$ around the identity intersects all $k_p$ of the $H$-infinite clusters with probability at least $1/2$. Letting $\mathscr{A}_r$ be the event in which the configuration obtained from $G_p$ by opening every edge in $B_r$ has a unique $H$-infinite cluster, it follows that $\mathscr{A}_r$ has positive probability. Since $\mathscr{A}_r$ is independent of the status of edges in $B_r$, it follows that there is a unique $H$-infinite with positive probability, contradicting the assumption that there are exactly $k_p>2$ such clusters almost surely.
\end{proof}

\noindent \textbf{Monotonicity of the unique infinite cluster.} Our first main relativized version of a classical theorem in percolation concerns the non-decay of the two-point function $\tau(x,y):=\mathbb{P}_p(x\leftrightarrow y)$ and uniqueness of the infinite cluster.
The classical version of this theorem was proven by Lyons and Schramm as a consequence of their \emph{indistinguishability theorem} \cite{LS99}; see also \cite{tang2018heavy} for an extension to arbitrary (possibly nonunimodular) transitive graphs.

\begin{theorem}[Uniqueness and long-range order]
\label{thm:monotone}
    Let $G$ be a Cayley graph of a group $\Gamma$ and $H$ be an infinite subgroup of $\Gamma$.  
For each $p\in [0,1]$, there is a unique $H$-infinite cluster in $G_p$ a.s.\ if and only if
$\inf_{x,y\in H}\tau_p(x,y)>0$. 
\end{theorem}

While we expect a relative version of the Lyons-Schramm indistinguishability theorem to hold, we instead give a shorter proof based on a theorem of Tim\'ar \cite{timar2006neighboring} stating that every pair of infinite clusters meet in at most finitely many places a.s. 

Noticing that $p\mapsto \inf_{x,y\in H}\tau_p(x,y)>0$ is increasing in $p$, we have \emph{uniqueness monotonicity} as an immediate consequence.

\begin{corollary}[Uniqueness Monotonicity]\label{cor:monotone}
Under the same setting as Theorem \ref{thm:monotone}, there is a unique $H$-infinite cluster a.s.\ for every $p>p_u(H;G)$.
\end{corollary}

 This corollary together with \cref{lem:Newman_Schulman} show that $p_c(H;G)$ and $p_u(H;G))$ are the only two points at which the number of $H$-infinite clusters changes.
The non-relativized version of uniqueness monotonicity was originally proven by H\"aggstr\"om, Peres, and Schonmann \cite{HaggPe99,HPS99,MR1676831}. 

\begin{remark}
The proof of \cref{thm:monotone} works to identify the a.s.\ uniqueness of the $H$-infinite cluster with the condition $\inf_{x,y\in H}\P(x\leftrightarrow y)>0$ for any \emph{finite-energy} (insertion and deletion tolerant) automorphism-invariant percolation process on $G$, which are the minimal conditions needed to apply Tim\'ar's theorem on cluster adjacencies \cite{timar2006neighboring}. Presumably only insertion-tolerance is needed.
\end{remark}

\medskip

\noindent
\textbf{Burton-Keane.}
As mentioned above, it is a classical result of Burton and Keane \cite{burton1989density} that percolation on any amenable transitive graph contains at most one infinite cluster almost surely. Our next main theorem is a relativized generalization of this theorem, applying also to amenable \emph{subgroups} of nonamenable groups. 

\begin{theorem}[Relative Burton-Keane]\label{thm:BK}
Let $G$ be a Cayley graph of a finitely generated group $\Gamma$, let $H$ be a subgroup of $\Gamma$, and let $p\in [0,1]$. If $H$ is amenable, then $G_p$ contains at most one $H$-infinite cluster a.s. In particular, $p_c(H;G)=p_u(H;G)$.
\end{theorem}

Note that this theorem can be applied even when $H$ is not finitely generated, in which case it is amenable if and only if all its finitely generated subgroups are; this is precisely the situation that arises in our analysis of lamplighter groups, where the group of lamp configurations is an infinitely generated abelian group.

\medskip

The proof of \cref{thm:BK} requires substantial modifications to the original Burton-Keane proof, parts of which do not readily generalize to the relative setting.
As with the classical Burton-Keane theorem, the relative Burton-Keane theorem extends immediately from Bernoulli bond percolation to arbitrary \emph{insertion-tolerant} automorphism-invariant percolation processes on the Cayley graph $G$, so that any such process contains at most one cluster having infinite intersection with the amenable subgroup~$H$. In fact, we will prove a much more general version of the theorem applying to insertion-tolerant percolation processes on \emph{hyperfinite locally unimodular random rooted graphs}, a notion we introduce in Section \ref{sec:background}. This lets us replace the subgroup appearing in the theorem with, say, the set of vertices visited by a doubly-infinite random walk; we will see in \cref{sec:random walks} that applying the relative Burton-Keane theorem in this setting allows us to solve a problem of Lyons and Schramm \cite{LS99} concerning the intersections of random walks with percolation clusters in the non-uniqueness regime.

\medskip

\noindent \textbf{Sharpness of the phase transition.} The next classical theorem we will prove is a relative version of the \emph{sharpness of the phase transition}.
The classical sharpness of the phase transition states that if $G$ is a transitive graph, then for each $p<p_c$ there exists a positive constant $c(p)$ such that
\[
\mathbb{P}_p(|K_o|\geq n) \leq e^{-c(p)n} \qquad \text{ for every $n\geq 1$,}
\]
where we let $K_x$ denote the cluster at $x$ and write $o$ for a fixed root vertex of $G$. This theorem is among the most important basic results of percolation theory.
 It was originally proven (in slightly weaker forms) in independent works of Menshikov \cite{MR852458} (who proved exponential decay of the \emph{radius} of $K_o$ below $p_c$) and Aizenman and Barsky \cite{aizenman1987sharpness} (who proved \emph{finite expectation} of $|K_o|$ below $p_c$). The strong form of the theorem stated above was first proven by Aizenman and Newman \cite{MR762034}, who proved that $|K_o|$ has an exponential tail whenever it has finite expectation. More recently, several new proofs of this theorem have been found \cite{duminil2015new,MR3898174,1901.10363,vanneuville2023exponential}, each of which has different levels of generality and quantitative implications. Most notably,  Duminil-Copin, Raoufi, and Tassion \cite{MR3898174} developed a new perspective on sharpness based on randomized algorithms and the OSSS inequality \cite{o2005every} which allowed them to prove that the subcritical cluster radius has an exponential tail for certain dependent models including the random cluster model. A modification of their proof was given by the first author in \cite{1901.10363} that establishes the exponential tail for the \emph{volume} distribution and also works for long-range models.



\medskip

We now state our relative version of the sharpness of the phase transition. 

\begin{theorem}[Relative sharpness of the phase transition]\label{thm:weak sharpness}
Let $G$ be a Cayley graph of a finitely generated group $\Gamma$ with $H$ a subgroup of $\Gamma$. 
For each $p<p_c(H;G)$, there exists $c(p)>0$ such that $\mathbb{P}_{p}(|K_o\cap H|\geq n)\leq \exp(-cn)$ for every $n\geq 1$.
\end{theorem}

See also [ref] for a version of this theorem that does not require any transitivity assumptions.
 The proof of this theorem is a relativized adaptation of the OSSS-based proof of the sharpness of the phase transition given in \cite{1901.10363}; interestingly, we do not know how to prove it using any of the more classical approaches to sharpness, which do not seem well-suited to relativization.


\medskip

\noindent \textbf{Relating relative and global $p_u$ under $wq$-normality.}
In order to apply these theorems, it will be important to be able to relate the relative value of $p_u$ back to its global value under appropriate assumptions on the subgroup $H$. To this end, we recall that
a subgroup $H$ of a group $\Gamma$ is said to be \textbf{$s$-normal} if its left and right cosets satisfy $|s H \cap H s|=\infty$ for every $s\in \Gamma$. In particular, infinite normal subgroups are always $s$-normal. Geometrically, an $s$-normal subgroup $H$ of a finitely generated group $\Gamma$ has the property that any two left cosets of $H$ ``come close to each other infinitely often'' in the sense that, in any Cayley graph of $\Gamma$, there exist constant-width neighbourhoods of the two cosets that have infinite intersection. (Note that if $H$ were infinite and \emph{normal}, rather than merely $s$-normal, then every left coset of $H$ is \emph{contained} in a constant-width neighborhood of every other left coset, with constant depending on the choice of cosets. See the proof of \cref{cor:examples} for an example of an $s$-normal subgroup that is not normal.) The following proposition relates relative $p_u$ with global $p_u$ under the assumption of $s$-normality.

\begin{prop}\label{prop:relative pu}
Let $G$ be a Cayley graph of a finitely generated group $\Gamma$ and let $H_1 \subseteq H_2 \subseteq \Gamma$ be subgroups of $\Gamma$. If $H_1$ is $s$-normal in $H_2$ then, for each $p\in [0,1]$, $G_p$ contains a unique $H_1$-infinite cluster a.s. if and only if it contains a unique $H_2$-infinite cluster a.s.
In particular, $p_u(H_1;G)=p_u(H_2;G)$ and if $H \subseteq \Gamma$ is an $s$-normal subgroup of $\Gamma$ then $p_u(H;G)=p_u(G)$.
\end{prop}



In fact the identity of the two values of $p_u$ holds under much weaker assumptions than $s$-normality. Indeed, it extends to \emph{wq-normal} subgroups as introduced by Popa \cite[Definition 2.3]{popa2006some}:
An infinite subgroup $H\subset G$ is said to be \textbf{wq-normal} (a.k.a.\ weak q-normal) in $G$ if given any intermediate subgroup $H\subset I\subset G$, there exists some $g\in G \setminus I$ such that $gIg^{-1}\cap I$ is infinite. Equivalently (as explained in \cite{popa2006some}), $H$ is wq-normal if there exists a countable ordinal $\alpha$ and an ascending $\alpha$-chain of subgroups $\{H_\beta\}_{\beta<\alpha}$ such that $H_o=H$, $H_\alpha=G$, and $\cup_{\beta<\gamma} H_\beta$ is a $s$-normal subgroup of $H_\gamma$ for every $\gamma \leq \alpha$. (Taking $\alpha=1$ shows that $s$-normal subgroups are wq-normal.) Using the tools we develop in the proof of uniqueness monotonicity, we can easily strengthen Proposition \ref{prop:relative pu} to the following:

\begin{prop}\label{prop:relative pu under wq-normal}
Proposition 1.8 remains true if ``$s$-normal'' is replaced by the weaker condition of ``$wq$-normal''.
\end{prop}

Putting together everything we have discussed so far suggests a natural approach to proving non-uniqueness at $p_u$: proving the \emph{non-existence} of $H$-infinite clusters at $p_c(H;G)$ for a well-chosen \emph{amenable}, $wq$-normal subgroup $H$. Moreover, this non-existence statement might itself be proven by relativizing one of the proofs of non-existence at $p_c$ for an appropriate class of groups. This is precisely the framework that will be used to prove \cref{thm:lamp}: It suffices to prove that if $L=\bigoplus_{\gamma \in \Gamma} \Z_2$ denotes the lamp group, which is abelian and normal in $\Z_2\wr \Gamma$, then there are no $L$-infinite clusters at $p_c(L;G)$. This statement can then be proven by relativizing the proof of \cite{Hutchcroft2016944}, which requires sharpness of the phase transition as an input.

\subsection{Obstacles to uniqueness}

We now state the strongest and most general forms of our results concerning non-uniqueness at $p_u$. These results will be formulated in terms of \emph{obstacles to uniqueness}, which strengthen the property of a group having non-uniqueness at $p_u$ in all of a group's Cayley graphs.

\begin{defn}
We say that a countable group $\Gamma$ is an \textbf{obstacle to uniqueness} if, whenever $G$ is a Cayley graph of a finitely generated group $\tilde \Gamma$ containing $\Gamma$ as a subgroup, there is not a unique $\Gamma$-infinite cluster at $p_u(\Gamma,G)$. (By \cref{lem:Newman_Schulman}, this means that there is either no $\Gamma$-infinite cluster  or infinitely many $\Gamma$-infinite clusters at $p_u(\Gamma,G)$ a.s.) Note that if a finitely generated group $\Gamma$ is an obstacle to uniqueness then every Cayley graph of $\Gamma$ has nonuniqueness at $p_u$; this follows by taking $\tilde \Gamma=\Gamma$ in the definition. 
\end{defn}

 The main advantage of working with this definition, rather than the usual notion of non-uniqueness at $p_u$, is that is satisfies the following \emph{co-heredity} property. This property is an immediate corollary of \cref{prop:relative pu}.


\begin{corollary}[Co-heredity from $wq$-normal subgroups]\label{thm:s-normal}
If $H$ is an $wq$-normal subgroup of a countable group $\Gamma$ and $H$ is an obstacle to uniqueness then $\Gamma$ is an obstacle to uniqueness.
\end{corollary}

Thus, to prove that a group is an obstacle to uniqueness, it suffices to find an $wq$-normal subgroup that is an obstacle to uniqueness. 

\medskip


\noindent \textbf{Sufficient conditions for obstacles to uniqueness.} 
We now present several sufficient conditions for a group to be an obstacle to uniqueness. The first condition we present is the most novel, with its proof using all the relativized classical percolation theorems we presented in the previous subsection to relativize the proof of \cite{Hutchcroft2016944}.

\begin{theorem}[Amenable groups of exponential growth are obstacles to uniqueness]
\label{thm:amenable}Let $\Gamma $ be a countable group and suppose that $\Gamma$ contains a finite set $S$ and an amenable $wq$-normal subgroup $H$ satisfying the exponential growth condition $\limsup_{n\to\infty}|S^n \cap H|^{1/n}>1$.
Then $\Gamma$ is an obstacle to uniqueness. In particular, finitely generated amenable groups of exponential growth are obstacles to uniqueness.
\end{theorem}

Theorem \ref{thm:lamp} is an immediate corollary of this theorem, as the lamplighter group $\Z_2 \wr \Gamma$ contains the group of lamp configurations $\bigoplus_{\gamma\in\Gamma}\Z_2$
as an amenable normal subgroup of exponential growth. See \cref{cor:examples} for a more general statement.



\medskip

In addition to this theorem, we also note that various proofs of non-uniqueness at $p_u$ appearing in the literature in fact establish the stronger property that these groups are obstacles to uniqueness. In particular, nonamenable products and property (T) groups are both obstacles to uniqueness, strengthening results of \cite{MR1770624,LS99,gaboriau2016approximations}. 

\begin{theorem}[Nonamenable products are obstacles to uniqueness]
\label{thm:products_are_obstacles}
Let $\Gamma=H\times K$, where $H$ is a finitely generated nonamenable group and $K$ is an infinite, finitely generated group. Then $\Gamma$ is an obstacle to uniqueness.
\end{theorem}

\begin{theorem}[Property (T) groups are obstacles to uniqueness]
\label{thm:property_T_obstacle}
Every countable group with Kazhdan's property (T) is an obstacle to uniqueness. More generally, every countable group that has relative property (T) with respect to some infinite $wq$-normal subgroup is an obstacle to uniqueness.
\end{theorem}

\begin{proof}[Proof of \cref{thm:products_are_obstacles,thm:property_T_obstacle}]
Both theorems follow by exactly the same arguments used to prove non-uniqueness at $p_u$ for the same examples in  \cite{gaboriau2016approximations}. Indeed, both results rely on general criteria in which a measured equivalence relation cannot be approximated by strict subequivalence relations, and, using \cref{thm:monotone}, one can apply these to the equivalence relation generated by the restrictions of the $\operatorname{Cay}(\tilde \Gamma,S)$-clusters to $\Gamma$ in exactly the same way as they are applied to percolation clusters on the group $\Gamma$. We omit the details as they are tangential to the main point of the paper.
\end{proof}


Let us now note some further concrete examples of groups that are treated by our theorems:

\begin{corollary}
\label{cor:examples}
The following groups are obstacles to uniqueness:
\begin{enumerate}
\item Wreath products $L\wr \Gamma$, where $\Gamma$ is infinite and $L$ is amenable.
\item $\operatorname{SL}_n(\Z[x])$, where $n\geq 2$ and $x \in \R\setminus \Z$.
\item Thompson's groups $F$ and $T$.
\end{enumerate}
\end{corollary}

It is a consequence of this corollary that all the listed examples have $p_u<1$. This was proven for wreath products with \emph{finite} lamp groups by Lyons and Schramm \cite[Corollary 6.8]{LS99}, but appears to be a new result for examples with infinite amenable lamp groups such as $\Z \wr \Z$. (Presumably these groups can be proven to have infinite internal isoperimetric dimension and hence also to have $p_u<1$ via the results of \cite{hutchcroft2021non}.)

\begin{proof}[Proof of \cref{cor:examples}]
Item 1: The case of this claim in which $\Gamma$ has an infinite finitely generated subgroup follows immediately from Theorem \ref{thm:amenable}. The general case follows by similar reasoning but where we use the metric associated to the finite generating set of the supergroup $\tilde \Gamma$ instead of that associated directly to $\Gamma$.
Item 2: As explained in \cite{bader2014weak}, the subgroup of upper-triangular matrices in $\operatorname{SL}_n(\Z[x])$ is solvable, hence amenable, and $s$-normal. (The group of upper-triangular matrices is \emph{not} $s$-normal in $\operatorname{SL}_2(\Z)$, which is virtually isomorphic to the free product $(\Z/2\Z)*(\Z/3\Z)$ and therefore has $p_u=1$.) This subgroup is also easily seen to have exponential growth. (Note that when $n>2$ this group has property (T) and hence can also be treated by \cref{thm:property_T_obstacle}.)
Item 3: As noted in \cite[Proof of Theorem 1.8]{bader2014weak}, Thompson's groups $F$ and $T$ both have $s$-normal subgroups isomorphic to $F^2$. Although it is a famous open problem whether $F$ and $T$ are amenable or not, they certainly have exponential volume growth. Thus, this $F^2$ subgroup is either an amenable group of exponential growth or a nonamenable product, and is an obstacle to uniqueness in either case.
\end{proof}

  \begin{remark}
Here is a further important setting in which amenable normal subgroups of exponential growth arise naturally. Consider a semidirect product $N \rtimes_\phi H$, where $\phi$ is a homomorphism from $H$ to $\operatorname{Aut}(N)$. If the image $\phi(H)$ has exponential growth, then $N$ has exponential growth as a subset of $N \rtimes_\phi H$. In many examples of interest this occurs with $N$ amenable and $\phi(H)$ nonamenable (and perhaps isomorphic to $H$), as is the case in e.g.\ $\Z^d \rtimes \operatorname{SL}_d(\Z)$. Such groups are obstacles to uniqueness by \cref{thm:amenable}. (The group $\Z^d \rtimes \operatorname{SL}_d(\Z)$ also has relative property (T), and could be treated using \cref{thm:property_T_obstacle}.) One reason to highlight constructions of this form is that they play a central role in the classification of Lie groups. Indeed, Levi's decomposition theorem states that every finite-dimensional, simply-connected Lie group can be written as a semi-direct product $R \rtimes G$, where $G$ is semisimple and $R$ is solvable \cite[Chapter 13.18]{varadarajan2013lie}.  Since the image of $G$ in $\operatorname{Aut}(R)$ is also semisimple, it is nonamenable whenever it is nontrivial. On the other hand, if the image is trivial but $R$ and $G$ are not then the group has a direct product structure and can be treated via other methods. These considerations make the methods of this paper applicable to prove non-uniqueness of the infinite cluster at the uniqueness threshold for percolation on all simply-connected Lie groups that are neither semisimple nor nilpotent, constituting a major step towards a general solution of the problem for Lie groups. We plan to develop the theory of percolation on Lie groups further in future work.
  \end{remark}

\medskip

\noindent \textbf{A bold conjecture.}
In \cite[Theorem 5.6]{peterson2011group}, Peterson and Thom proved that if $H$ is $wq$-normal in $G$, then $H$ having vanishing first $\ell^2$-Betti number implies $G$ having vanishing first $\ell^2$-Betti number. See \cite[Chapter 10.8]{LP:book} for an account of these numbers suitable for probabilists and \cite{bader2014weak} for extensions to higher $\ell^2$-Betti numbers. Thus, obstacles to uniqueness have exactly the same co-heredity property as groups of vanishing first $\ell^2$-Betti number. Moreover, in all examples \emph{other than the integers} where (non-)uniqueness at $p_u$ has been analyzed, uniqueness at $p_u$ has held if and only if the group has positive first $\ell^2$-Betti number. This supports the following conjecture:

\begin{conjecture}
\label{conjecture:betti_number}
An infinite, finitely generated group is an obstacle to uniqueness if and only if it is not virtually isomorphic to $\Z$ and has vanishing first $\ell^2$-Betti number.
\end{conjecture}

\begin{remark}
Since infinite amenable groups always have vanishing first $\ell^2$-Betti number,
\cref{conjecture:betti_number} implies the conjecture that there is no percolation at $p_c$ on any Cayley graph of superlinear growth: This conjecture has already been verified for nonamenable groups, while every finitely generated amenable group that is an obstacle to uniqueness must have no infinite clusters at criticality on any of its Cayley graphs. More concretely, if $\Z^n$ is an obstacle to uniqueness then every Cayley graph of $\Z^m$ has no infinite clusters at criticality for every $m\geq n$.
\end{remark}

\begin{remark}
Benjamini and Schramm (see \cite[Conjecture 6.1]{LPS06}) conjectured that a nonamenable Cayley graph has uniqueness at $p_u$ if and only if its \emph{free minimal spanning forest} (FMSF) is connected, which can be the case only if the group has positive first $\ell^2$ Betti number. One direction of this conjecture was proven by Lyons, Peres, and Schramm \cite{LPS06}: If the FMSF is connected then there is uniqueness at $p_u$. The converse is however \emph{false}: The free product of a non-treeable group (e.g. $\operatorname{SL}_3(\Z)$) with $\Z_2$ is non-treeable, so cannot have a connected FMSF, but has $p_u=1$ and hence uniqueness at $p_u$. The conjecture remains open for groups with $p_u<1$.
\end{remark}


\noindent
\textbf{Organization}  
The paper is organized as follows. In Section \ref{sec:background}, we give a introduction to unimodular random graphs. In Section \ref{sec:ends} we study the number of ends in locally
unimodular random graphs and prove the relative Burton-Keane theorem. In Section \ref{sec:cluster pu}, we prove a sufficient condition for a cluster to be the unique cluster having infinite intersection with a subgroup, from which we derive relative uniqueness monotonicity (\cref{thm:monotone}) and co-heredity from $wq$-normal subgroups (\cref{prop:relative pu under wq-normal}).
In Section \ref{sec:sharp transition}, we study relative sharpness of the phase transition and prove Theorem \ref{thm:weak sharpness}. In Section \ref{sec:obstacles} we prove that amenable groups of exponential growth are obstacles to uniqueness. Finally, in Section \ref{sec:random walks}, we show that the machinery we have developed can also be used to solve a problem of Lyons and Schramm \cite{LS99} regarding the intersection of random walks with percolation clusters in the non-uniqueness regime.

\section{Background}\label{sec:background}

Several of our results, including the relative Burton--Keane theorem, can be formulated more generally as results about \emph{locally unimodular random rooted graphs}, a notion introduced in \cite{hutchcroft2020non}, where the appropriate analogue of amenability is \emph{hyperfiniteness}. In this section we introduce all relevant notions needed to understand the statements and proofs.

\subsection{Unimodular transitive graphs}

Let $G$ be a locally finite transitive graph and consider its automorphism group $\operatorname{Aut}(G)$. For a transitive subgroup $\Gamma$ of $\operatorname{Aut}(G)$, we let the \textbf{modular function} be 
\[\Delta(u,v):=|\operatorname{Stab}_vu|/|\operatorname{Stab}_uv|,\]
where $\operatorname{Stab}_x$ is the stabilizer of $x$ in $\Gamma$ and $\operatorname{Stab}_xy$ is the orbit of $y$ under $\operatorname{Stab}_x$. The modular function is useful because for any function $F:G\times G\rightarrow [0,\infty]$ that is $\Gamma$-invariant, we have the \textbf{tilted mass-transport principle} (see \cite{Hutchcroftnonunimodularperc} for example):
\begin{equation}
\label{eq:tiltedMTP}
\sum_x F(o,x) = \sum_x F(x,o)\Delta(o,x).
\end{equation}
If $\Delta\equiv 1$, we say $\Gamma$ is \textbf{unimodular}; otherwise, $\Gamma$ is said to be \textbf{nonunimodular}. If there exists a unimodular transitive subgroup $\Gamma$ of $\operatorname{Aut}(G)$, then $G$ is said to be a \textbf{unimodular transitive graph}. Cayley graphs and amenable transitive graphs are always unimodular \cite{MR1082868}.

\subsection{Unimodular random graphs}

The notion of unimodular random graphs was first introduced by \cite{BeSc} as a generalization of Cayley graphs and later systematically studied by \cite{AL07}. We refer the reader to the lecture notes \cite{CurienNotes} for a detailed introduction to the subject.

A\textbf{\ rooted graph} $\left( G,\rho \right) $ is a countable locally
finite graph $\left( V,E\right) $ with a distinguished vertex $\rho $, which
we call the \textbf{root}. An isomorphism $\varphi :(G,\rho )\rightarrow
(G^{\prime },\rho ^{\prime })$ between two rooted graphs is an isomorphism
between $G$ and $G^{\prime }$ that maps $\rho $ to $\rho ^{\prime }$. We
denote by $\mathcal{G}_{\bullet }$ the set of isomorphism classes of rooted graphs. The set
$\mathcal{G}_{\bullet }$ is endowed with the local topology, which is induced by
the metric%
\begin{equation*}
d_\mathrm{loc}\left( (G,\rho ),(G^{\prime },\rho ^{\prime })\right) =e^{-r},
\end{equation*}%
where $r$ is the maximum radius such that the ball $B_{r}(G,\rho )$ centered
in $\rho $ of radius $r$ in $G$ is isomorphic to $B_{r}(G^{\prime
},r^{\prime })$. This metric makes $\mathcal{G}_{\bullet }$ a Polish space.
A \textbf{random rooted graph} is a random variable taking values in $\left( G,\rho
\right) $.
We similarly define a \textbf{doubly rooted graph} to be a graph together
with an ordered pair of distinguished vertices and let $\mathcal{G}_{\bullet
\bullet }$ denote the set of isomorphism classes of doubly rooted graphs,
which is endowed with the natural version of the local topology for doubly-rooted graphs. 

A\textbf{\ mass transport function} is a Borel function $f:\mathcal{G}%
_{\bullet \bullet }\rightarrow \lbrack 0,\infty ]$. Such functions can also be thought of intuitively as functions from doubly-rooted graphs to $[0,\infty]$  that are invariant
under isomorphisms. There is a mild technical difficulty making this precise as ``the set of all doubly-rooted graphs'' is not really a well-defined set. A \textbf{unimodular random graph} is a random rooted graph satisfying the
s-mass-transport principle, meaning that
\begin{equation*}
\mathbb{E}\left[ \sum_{v\in V}f(G,\rho ,v)\right] =\mathbb{E}\left[
\sum_{v\in V}f(G,v,\rho )\right] 
\end{equation*}
for every mass-transport function $f:\mathcal{G}%
_{\bullet \bullet }\rightarrow \lbrack 0,\infty )$.
With this definition, all the Cayley graphs and Benjamini-Schramm limits of finite graphs are unimodular \cite[Section 2]{CurienNotes}.

\subsection{Hyperfiniteness}

We now introduce the notion of \emph{hyperfiniteness} for unimodular random rooted graphs, which generalized the notion of amenability of groups.
A \textbf{finitary partition} of a graph $G$ is a partition of its vertices such that each cell only has finitely many vertices. An \textbf{exhaustive partition sequence} of a graph $G$ is defined to be a sequence of finitary partitions $(\Pi_i)_i$ such that each $\Pi _{i+1}$ is coarser than $\Pi _{i}$ and any two points $x,y$
are eventually included in the same cell of $\Pi _{i}$. 
 
A unimodular random rooted graph $(G,\rho)$ is said to be \textbf{hyperfinite} if (after passing to a larger probability space if necessary) there exists an exhaustive partition sequence $(\Pi_i)_i$  such that $(G,\rho,(\Pi_i)_i)$ is unimodular in the sense that for any Borel function $f$ on isomorphism classes of tuples\footnote{The space of isomorphism classes of tuples consisting of a graph, two points,
and a sequence of partitions is endowed with a similar local topology as before.} $\left(
G,x,y,(\Pi _{i})_i \right)$,
the tuple $\left( G,\rho ,(\Pi _{i})_i\right) $ satisfies the mass-transport
principle:
\begin{equation*}
\mathbb{E}\left[ \sum_{v\in V}f(G,\rho ,v,(\Pi _{i})_i)\right] =\mathbb{E}%
\left[ \sum_{v\in V}f(G,v,\rho ,(\Pi _{i})_i)\right].
\end{equation*}
It should be noted that hyperfiniteness is a property of the \emph{law} of the random graph, rather than being defined pointwise. If we consider the Cayley graph of a finitely-generated group as a unimodular random rooted graph, it is hyperfinite if and only if the group is amenable \cite{BLPS99}. Further background on hyperfiniteness can be found in Section 3 of \cite{unimodular2} and Section 8 of \cite{AL07}.

\subsection{Local unimodularity}
We now introduce locally unimodular random rooted graphs, following \cite{hutchcroft2020non}; these objects generalize pairs consisting of a Cayley graph and a distinguished subgroup in a similar way to how unimodular random graphs generalize Cayley graphs. We define a \textbf{locally
rooted graph} to be a tuple $\left( G,A,\rho \right) $ where $\rho \in A\subset V(G)$ and let $\mathcal{G}_{\bullet
}^{\diamond }$ denote the set of isomorphism classes of locally rooted graphs. Similarly, a\textbf{\ locally doubly-rooted graph} is a tuple $%
\left( G,A,\rho,\rho ^{\prime }\right) $, where $\rho ^{\prime },\rho \in
A\subset G$, and the set of isomorphism classes of locally doubly-rooted graphs is denoted by $\mathcal{G}_{\bullet \bullet }^{\diamond
}$. We say a random locally rooted graph $(G,A,\rho)$ is \textbf{locally} 
\textbf{unimodular} if%
\begin{equation*}
\mathbb{E}\left[ \sum_{v\in A}f(G,A,\rho ,v)\right] =\mathbb{E}\left[
\sum_{v\in A}f(G,A,v,\rho )\right]
\end{equation*}%
for every Borel function $f:\mathcal{G}_{\bullet \bullet }^{\diamond
}\rightarrow [0,\infty ]$. In particular, if $A$ is a finite set of vertices in a connected, locally finite graph $G$ and $\rho$ is chosen uniformly in $A$, then $(G,A,\rho)$ is a locally unimodular graph. Another example occurs when $A$ is a subgroup of a group $G$ and $\rho=o$ is the identity element: since every element of $A$ acts bijectively on $A$ via left multiplication, and this defines an automorphism of $(G,A)$, we have that
\begin{equation*}
\mathbb{E}\left[ \sum_{v\in A}f(G,A,o ,v)\right] 
=\mathbb{E}\left[ \sum_{v\in A}f(v^{-1}G,v^{-1}A,v^{-1},o)\right] =\mathbb{E}\left[
\sum_{v\in A}f(G,A,v,o)\right]
\end{equation*}%
for every Borel function $f:\mathcal{G}_{\bullet \bullet }^{\diamond
}\rightarrow [0,\infty ]$ as required.
Moreover, the class of locally unimodular random rooted graphs is closed under Benjamini-Schramm limits by the same reasoning as for unimodular random graphs.

We say a random locally rooted graph $\left( G,A,\rho \right) $ is \textbf{hyperfinite locally unimodular} if there exists an exhaustive partition sequence $(\Pi_i)_i$ of $A$ such that $(G,A,\rho,(\Pi_i)_i)$ is locally unimodular in the sense that the mass-transport principle
\begin{equation}\label{eq:MTP, local hyperfinite}
\mathbb{E}\left[ \sum_{v\in A}f(G,A,\rho ,v,(\Pi _{i})_i)\right] =\mathbb{E}%
\left[ \sum_{v\in A}f(G,A,v,\rho ,(\Pi _{i})_i)\right] 
\end{equation}%
holds for any non-negative Borel function $f$ defined on the space of isomorphism classes of locally doubly-rooted graphs decorated by a sequence of partitions. 



\begin{lemma}
\label{lem:hyperfiniteness_two_ways}
In the definition of hyperfiniteness, we can equivalently take $(\Pi_i)_i$ to be an exhaustive partition sequence of the vertex set of $G$, rather than of $A$.
\end{lemma}

\begin{proof}[Proof of \cref{lem:hyperfiniteness_two_ways}]
If we start with an exhaustive partition sequence of the vertex set $V$ of $G$, we can restrict this partition sequence to $A$ to get an exhaustive partition sequence of $A$, and this clearly preserves unimodularity. Conversely, suppose that we have an exhaustive partition sequence $(\Pi_i)_i$ of $A$ such that $(G,A,\rho,(\Pi_i)_i)$ is unimodular.
Conditional on $(G,A,\rho,(\Pi_i)_i)$, for each $x\in V\setminus A$, let $a(x)\in A$ be uniformly chosen from $\{a\in A, d(x,a)=d(x,A)\}$, making the choices independent across $V\setminus A$. 
For each $i$, let $(\tilde \Pi_i)_i$ be the partition of $V$ consisting of the sets $a^{-1}(E)\cap \{x\in V:d(x,A)<i\}$ for $E\in \Pi_i$ and all the singleton sets $\{x\}$ for $d(x,A)\geq i$. It is easily checked that the resulting sequence of partitions $(\tilde \Pi_i)_i$ is an exhaustive partition sequence of $V$ making the tuple $(G,A,\rho,(\tilde \Pi_i)_i)$ unimodular.
\end{proof}

\begin{lemma}\label{lm:amenable-hyperfinite}
Let $G$ be a Cayley graph of a finitely generated group $\Gamma$ and $H$ be a (not necessarily finitely generated) subgroup of $\Gamma$. Then the locally unimodular random rooted graph $(G,H,o)$ is hyperfinite if and only if $H$ is amenable.
\end{lemma}

\begin{proof}[Proof of \cref{lm:amenable-hyperfinite}]
This is essentially a consequence of \cite[Theorem 5.1]{BLPS99}. The argument needed to deduce the existence of unimodular random exhaustive partition sequences from the existence of unimodular finitary percolation processes with arbitrarily high marginals is identical to that given in e.g. \cite[Section 3.3]{unimodular2}.
\qedhere

\end{proof}


\section{Ends in locally unimodular graphs and relative Burton-Keane}
\label{sec:ends} 

In this section we prove the relative Burton-Keane theorem for hyperfinite locally unimodular random rooted graphs. Given a locally unimodular random rooted graph $(G,A,\rho)$, we say that a random subgraph $\omega$ of $G$ is a \textbf{locally unimodular percolation process} on $(G,A,\rho)$ if the tuple $(G,A,\rho,\omega)$ is locally unimodular in the sense that it satisfies the mass-transport principle. We also recall that a random subgraph $\omega$ of a graph $G$ is said to be \textbf{insertion-tolerant} if the law of $\omega \cup F$ is absolutely continuous with respect to that of $\omega$ for each finite set of edges $F$.

\begin{theorem}[Relative Burton-Keane]\label{cor:unique infinite cluster}
Suppose that $\left( G,A,\rho \right) $ is a hyperfinite locally
unimodular random rooted graph and let $\omega$ be a locally unimodular percolation process on $(G,A,\rho)$.
If the conditional law of $\omega$ given $(G,A,\rho)$ is insertion-tolerant, then $\omega$ contains at most one $A$-infinite cluster almost surely.

\end{theorem}
Note that Theorem \ref{thm:BK} is an immediate consequence of the corollary, as Lemma \ref{lm:amenable-hyperfinite} gives $(G,H,o)$ is a hyperfinite locally unimodular random rooted graph and Bernoulli bond percolation with $p>0$ defines an insertion-tolerant, locally unimodular percolation process on $(G,H,o)$.

To prove this theorem, we first prove a more general theorem about ends in hyperfinite locally unimodular random rooted graphs, from which the theorem will follow as an easy corollary.
Recall that a graph $G$ is said to be $k$-\textbf{ended}
 if deleting finitely many vertices from $G$ results in a maximum of $k$ infinite connected components. Unimodular random graphs already have strong regularity about the number of ends they can have: the number of ends of a unimodular random graph $(G,o)$ necessarily belongs to $%
\left\{ 0,1,2,\infty \right\}$ \cite[Theorem 14]{CurienNotes}. Furthermore, if $(G,o)$ is hyperfinite a.s., then $G$ has at most two ends a.s \cite{AL07}, with the case of exactly two ends forcing $(G,o)$ to be hyperfinite \cite{unimodular2}. 

We will instead want to study the number of ends of a graph \emph{relative to a hyperfinite subgraph}, in the following sense. 
Given a graph $G$ and a set of vertices $A\subset V(G)$, we say $A$ is $k$%
\textbf{-ended in $G$} if deleting finitely many vertices from $G$ results in a
maximum of $k$ connected components each having infinite intersection with $A$. If we take $A=V(G)$, then this coincides with our usual definition of ends. The main technical result of this section is as follows. 

\begin{theorem}\label{thm:ends}
Suppose that $\left( G,A,\rho \right) $ is a hyperfinite locally
unimodular random rooted graph. Then $A$ is either
finite, one-ended, or two-ended in $G$ a.s.
\end{theorem}


Before proving this theorem, let us first see how it implies \cref{cor:unique infinite cluster}.  

\begin{proof}[Proof of \cref{cor:unique infinite cluster} given \cref{thm:ends}]
Let $(\Pi_i)_i$ be an exhaustive partition of $A$ for which the tuple $(G,A,\rho,(\Pi_i)_i)$ is unimodular (i.e.\ satisfies the mass-transport principle). We first note that $(K_\rho,A\cap K_\rho,\rho)$ is a hyperfinite locally unimodular random rooted graph. Indeed, the proof of this fact is identical to that of previous results showing restrictions to percolation clusters preserve unimodularity given in \cite[Lemma 3.1]{unimodular2} and \cite[Proposition 10]{CurienNotes}. (In particular, we can construct an exhaustive partition sequence by sampling $(\Pi_i)_i$ and the percolation configuration conditionally independently given $(G,A,\rho)$ and then restricting the partition to $K$.) Thus, \cref{thm:ends} implies that $A \cap K_\rho$ is either finite, one-ended, or two-ended in $G$ a.s. Moreover, the relative Newman-Schulman theorem also applies at this level of generality (with the same proof as \cref{lem:Newman_Schulman}; we may assume that $(G,A,\rho)$ is ergodic by first passing to an ergodic decomposition if necessary) to show that the number of $A$-infinite clusters belongs to $\{0,1,\infty\}$ almost surely. Suppose for contradiction that there are infinitely many $A$-infinite clusters with positive probability. By continuity of measure, there exists $r<\infty$ such that the ball of radius $r$ around $\rho$ in $G$ intersects at least three $A$-infinite clusters with positive probability. Using insertion tolerance, we deduce that $A \cap K_\rho$ is at least three-ended in $K_\rho$ with positive probability, contradicting \cref{thm:ends}.
\end{proof}

We now prove \cref{thm:ends}. We will deduce this theorem as a corollary of the following theorem of the first author \cite{hutchcroft2020non}, which is based on the \emph{magic lemma} of Benjamini and Schramm \cite{BeSc}.

\begin{theorem}[{\!\!\cite[Theorem 3.3]{hutchcroft2020non}}]\label{thm:tree ends}
Let $(T_n,A_n,o_n)_{n\geq 1}$ be a sequence of locally unimodular random rooted \emph{trees} converging in distribution to some random variable $(T,A,o)$ as $n\rightarrow \infty$. If $A_n$ is finite almost surely for every $n\geq 1$, then $A$ is either finite, one-ended, or two-ended in $T$ almost surely.     
\end{theorem}

This theorem has the following immediate corollary.

\begin{corollary}\label{cor:tree_ends_hyperfinite}
Let $(T,A,o)$ be a locally unimodular random rooted tree. If $(T,A,o)$ is hyperfinite then $A$ is either finite, one-ended, or two-ended in $T$ almost surely.     
\end{corollary}

\begin{remark}
It can be shown conversely that if $A$ is either finite, one-ended, or two-ended in $T$ almost surely then $(T,A,o)$ is hyperfinite.
\end{remark}

\begin{proof}[Proof of \cref{cor:tree_ends_hyperfinite}]
Apply \cref{thm:tree ends} to the sequence $(T,A_n,o)_{n\geq 1}$ formed by setting $A_n$ to be the component of $o$ in the partition $\Pi_n$; it is easily seen that $(T,A_n,o)$ is locally unimodular and converges to $(T,A,o)$ as $n\to \infty$.
\end{proof}

The proof of Theorem \ref{thm:ends} can be summarized as follows. We first find a spanning tree $T$ of $G$ for which $(T,A,o)$ is a hyperfinite locally unimodular random rooted graph and then apply Theorem \ref{thm:tree ends} to get $A$ is at most $2$-ended in $T$. As $T$ is a spanning tree of $G$, any end in $G$ corresponds to at least one end in $T$, so that $A$ is at most 2-ended in $G$ as desired.

\begin{proof}[Proof of Theorem \ref{thm:ends}]
By \cref{lem:hyperfiniteness_two_ways}, we may take an exhaustive partition sequence $(\Pi_i)_i$ of the vertex set of $G$ such that $(G,A,\rho,(\Pi_i)_i)$ satisfies the mass-transport principle. If we define $\Lambda_i$ to be the refinement of $\Pi_i$ defined by setting two vertices $x$ and $y$ to belong to the same class of $\Lambda_i$ if they are in the same class of $\Pi_i$ and there exists a path from $x$ to $y$ using only vertices in this class of $\Pi_i$, then $(\Lambda_i)_i$ is also an exhaustive partition sequence of $V$ such that $(G,A,\rho,(\Lambda_i)_i)$ is unimodular, and has the additional property that every equivalence class of $\Lambda_i$ is connected in $G$ for each $i\geq 1$.

We use this partition sequence to inductively construct a spanning tree $T$ for $G$, following the proof of
\cite[Theorem 5.3]{BLPS99}. First, for each cell of $\Lambda _{1}$, take a
uniformly chosen spanning tree of the subgraph of $G$ induced by each component of $\Lambda_1$, choosing these spanning trees independently for each such component, and let $T_{1}$ be the union of these spanning trees. Inductively,
suppose that $n\geq 1$ and that we have constructed a spanning forest $T_n$ of $G$ with the property that $T_n$ restricts to a spanning tree of each component of $\Lambda_n$ and has no edges between different components of $\Lambda_n$.
 Recall that each cell of $\Lambda _{n+1}$ is a union of cells in $\Lambda _{n}$.
In each cell $E$ of $\Lambda_{n+1}$, we choose a spanning tree of the subgraph of $G$ induced by $E$ uniformly at random from among those trees which contain all the trees in $T_n$ that span one of the cells of $\Lambda_n$ contained in $E$. We make these choices independently at random for the different classes of $\Lambda_{n+1}$, and take the forest $T_{n+1}$ to be the union of all these trees.
Since every pair of vertices in $G$ eventually belongs to the same class of $\Lambda_n$, the union
 $T:=\cup T_{n}$ is a spanning tree of $G$. Since the procedure used to construct the law of the spanning tree $T$ from the datum $(G,A,\rho,(\Lambda_i)_i)$ is automorpism-equivariant, it follows by standard arguments that $(T,A,\rho,(\Lambda_i)_i)$ is locally unimodular and hence that $(T,A,\rho)$ is a hyperfinite locally unimodular random rooted graph.
%
 Applying Theorem \ref{thm:tree ends}, it follows that $A$ is at most $2$-ended in $T$ a.s. Since $T$
is a spanning tree of $G$, any end in $G$ corresponds to at least one end in 
$T$, and it follows that $A$ is at most $2$-ended in $G$ a.s.\ also.
\end{proof}

\section{Relative uniqueness monotonicity}\label{sec:cluster pu}

In this section, we continue building the theory of subgroup relativization
by proving two results about the relative uniqueness threshold: Proposition %
\ref{thm:monotone} and Proposition \ref{prop:relative pu}.

\subsection{Uniqueness via invariant random choice of a single cluster}

In this section we prove the following lemma, which is a key step in the the proof of uniqueness monotonicity.

\begin{lemma}
\label{lemma:infinite cluster} 
Let $G$ be a Cayley graph of a group $\Gamma $ and $H$ be a subgroup of $\Gamma$. Suppose that for some $p\in (0,1)$ there exists a random pair $(G_p,K)$ such that $G_p$ is distributed as Bernoulli-$p$ bond percolation on $G$, the set $K$ is almost surely a cluster of $G$, and the law of $(G_p,K)$ is invariant under the action of $H$ on $\{0,1\}^E \times \{0,1\}^V$.
%
   Then $K$ is the unique $H$-infinite cluster 
   $\mathbb{%
P}_{p}$-a.s. 
\end{lemma}

It is an implicit assumption of this lemma that $K$ is non-empty almost surely. Note however that $K\cap H$ is \emph{not} assumed to be non-empty. We will usually think of the cluster $K$ as being chosen from among the clusters of $G_p$ using some $H$-invariant procedure (that may rely on additional randomness).

\medskip

We will deduce \cref{lemma:infinite cluster} from the following theorem of Tim\'ar:

\begin{theorem}
\label{thm:Timar}
Let $G$ be a locally finite  unimodular transitive graph and let $p\in [0,1]$. Almost surley, for every pair of distinct clusters $(K_1,K_2)$ in $G_p$, the set of edges of $G$ with one endpoint in $K_1$ and the other in $K_2$ is finite.
\end{theorem}

We begin by noting that \cref{thm:Timar} implies the following slightly stronger theorem.

\begin{corollary}
\label{cor:Timar_r}
Let $G$ be a locally finite  unimodular transitive graph and let $p\in [0,1]$. Almost surley, for every $r\geq 1$ and every pair of distinct clusters $(K_1,K_2)$ in $G_p$, the set of pairs $(x,y)\in K_1\times K_2$ with graph distance at most $r$ is finite.
\end{corollary}

\begin{proof}[Proof of \cref{cor:Timar_r}]
Suppose not, so that there exists $r$ finite such that, with positive probability, there exist two distinct clusters $K_1$ and $K_2$ such that the set $\{(x,y)\in K_1\times K_2 : d(x,y)\leq r\}$ is infinite. Let $\Omega$ be the event that every pair of distinct clusters of $G_p$ have at most finitely many edges of $G$ incident to both of them, which has probability $1$ by \cref{thm:Timar}.
By insertion tolerance, there exist with positive probability two such clusters that touch in at least one place, i.e., have an edge of $G$ with one endpoint in $K_1$ and the other in $K_2$.
By translation-invariance, there exists a neighbour $a$ of $o$ such that, with positive probability, the clusters of $o$ and $a$ are distinct and $\{(x,y)\in K_o\times K_a : d(x,y)\leq r\}$ is infinite. Let $\mathscr{E}$ denote this event.
 For each pair of vertices $x$ and $y$ in $G$, let $\mathscr{T}(x,y)$ denote the set of vertices in the cluster $K_x$ that have a neighbour in $G$ that is in $K_y$ but not $K_x$, so that $\mathscr{T}(x,y)=\emptyset$ if $x$ and $y$ are connected. 
If we can define a mass-transport function by
 \[
f(G_p,x,y) = \sum_{z\in B(x,r), \mathscr{T}(x,z)\neq \emptyset}\frac{\mathbbm{1}(y\in \mathscr{T}(x,z))}{|\mathscr{T}(x,z)|},
 \]
then this function trivially satisfies $\sum_{y}f(G_p,x,y)\leq |B_r(o)|<\infty$. On the other hand, if the event $\mathscr{E} \cap \Omega$ holds then $\sum_x f(G_p,x,o)=\infty$, contradicting the mass-transport principle.
\end{proof}

We now apply \cref{cor:Timar_r} to prove \cref{lemma:infinite cluster}.

\begin{proof}[Proof of \cref{lemma:infinite cluster}]
We first show that there exists an $H$-infinite cluster a.s. By continuity of measure, there exists $r$ large enough such that $\mathbb{P}_p(B_r(o)\cap K\neq \emptyset)\geq \frac{3}{4}$. Since the law of $K$ is $H$-invariant,  we have $\mathbb{P}_p(B_r(x)\cap K\neq \emptyset)\geq \frac{3}{4}$ for every $x\in H$ and hence that
\[\mathbb{P}_p(B_r(x)\text{ is connected to }B_r(y))\geq \mathbb{P}_p(B_r(x)\cap K\neq \emptyset\text{ and }B_r(y)\cap K\neq \emptyset)\geq \frac{1}{2}\]
for every $x,y\in H$.
Conditional on $B_r(x)$ being connected to $B_r(y)$, if all the edges in $B_r(x)$ and $B_r(y)$ are flipped open, then $x$ is connected to $y$. Hence it follows from insertion tolerance that there is some constant $c>0$ such that $\tau_p(x,y)>c$ for every $x,y\in H$. It follows by Fatou's lemma that
\[\mathbb{P}_p(|K_o\cap H|=\infty)=\mathbb{P}_p(o \text{ is connected to infinitely many }h\in H)\geq c>0,\]
and since the existence of a $H$-infinite cluster is an $H$-invariant event 
 probability that there exists an $H$-infinite cluster must be $1$ as claimed.

We now prove that $K$ is the unique $H$-infinite cluster a.s.
Suppose for contradiction $K$ is not the only $H$-infinite cluster, so that with positive probability the origin $o$ belongs to an $H$-infinite cluster that is not equal to $K$. (As far as we know at this point of the proof, $K$ might not even intersect $H$ at all.) By continuity of measure, there exists $r<\infty$ such that with positive probability, the origin belongs to an $H$-infinite cluster distinct from $K$, but $K$ intersects the ball of radius $r$ around the origin. For each $x\in H \setminus K$, let $\mathscr{T}(x)$ denote the set of points $y\in K_x \cap H$ such that $B_r(y)\cap K \neq \emptyset$, which is finite a.s.\ by \cref{cor:Timar_r}. Consider the mass-transport function
\[
f(G_p,K,x,y) = \frac{\mathbbm{1}(x\notin K, y \in \mathscr{T}(x))}{|\mathscr{T}(x)|}.
\]
Applying the mass-transport principle, which is valid since $(G_p,K)$ is $H$-invariant, we obtain that
\begin{multline*}
\P(o\notin K, \mathscr{T}(o)\neq \emptyset) = \E\left[ \sum_{x\in H} f(G_p,K,o,x)\right] \\= 
\E\left[ \sum_{x\in H} f(G_p,K,x,o)\right]=
\E\left[\frac{|K_o\cap H|}{|\mathscr{T}(o)|} \mathbbm{1}(o\in \mathscr{T}(o))\right],
\end{multline*}
where the final equality follows since $\mathscr{T}(x)=\mathscr{T}(o)$ for every $x\in H$ belonging to the same cluster as $o$. This yields a contradiction since the left hand side is at most $1$ while the right hand side is the expectation of a non-negative random variable that is infinite with positive probability.
\end{proof}

\cref{lemma:infinite cluster} can immediately be used to deduce Corollary \ref{cor:monotone} as follows. 
Suppose that there is a unique $H$-infinite cluster $K_{p}$ in $G_{p}$ a.s. Then the law of this
infinite cluster is $H$-invariant. For any $p'>p$, in the standard monotone coupling of $G_p$ with $G_{p'}$, the law of the infinite cluster in $G_{p'}$ containing $K_p$ is $H$-invariant. Invoking Lemma \ref{lemma:infinite
cluster}, we get there is a unique $H$-infinite cluster $K_{p^{\prime }}$ in $G_{p^{\prime }}$ a.s. 
Theorem \ref{thm:monotone}, which identifies $p_u$ with the infimum points such that connection probabilities do not decay, requires a further argument that is given in the next subsection.

\subsection{$H$-frequency of clusters and Proof of Theorem \protect\ref%
{thm:monotone}.}\label{subsec: H-freq}

In this section we conclude the proof of our uniqueness monotonicity theorem, \cref{thm:monotone}, by proving that if $\inf_{x,y\in H} \tau _{p}(x,y)>0$ then there is a unique $H$-infinite cluster (the other direction is an immediate consequence of Harris-FKG). By \cref{lemma:infinite cluster}, it suffices to prove that, under this condition, it is possible to choose an $H$-infinite cluster of $G_p$ in an $H$-invariant manner.
To do this, we relativize an idea due to Lyons and Schramm \cite[Section 4]{LS99} by
introducing the \textquotedblleft density\textquotedblright\ of a cluster
relative to $H$. We then argue that if $\inf \tau _{p}(x,y)>0$ then there must
exist a cluster with positive density, and by choosing a cluster of maximal density uniformly at random we obtain an $H$-invariant random choice of cluster as desired.

Let $\Gamma$ be a countable group and let $\left\{ Z_{n}\right\} $ be a random walk on $\Gamma$ with symmetric step distribution $\mu$, whose
(not necessarily finite) support generates $\Gamma$. (We will later take $\Gamma$ to be a subgroup $H$ of some larger group.) For each $x\in \Gamma$, let $\mathbf{P}_{x}$ denote the law
of the random walk $\left\{ Z_{n}\right\} $ starting from $x$. The \textbf{frequency} of a set $W$ in $\Gamma$, when it is defined, will measure the density of $W$ in $\Gamma$ via the proportion of time that $Z$ spends in $W$. To define this function, we first define its domain to be the set $\Omega(\Gamma,\mu)$ of subsets of $\Gamma$ given by
\begin{multline*}
\Omega(\Gamma,\mu):=\Bigl\{W \subseteq \Gamma : \text{there exists $\operatorname{Freq}_\mu(W)\in [0,1]$ such that} \\\lim_{n\to\infty} \frac{1}{n}\sum_{i=1}^n \mathbbm{1}(Z_n \in W) = \operatorname{Freq}_\mu(W) \text{ $\mathbf{P}_x$-a.s.\ for every $x\in \Gamma$}\Bigr\}.
\end{multline*}
For each set $W\in \Omega(\Gamma,\mu)$, its frequency $\operatorname{Freq}_\mu(W)\in[0,1]$ is defined to be the unique number such that $\lim_{n\to\infty} \frac{1}{n}\sum_{i=1}^n \mathbbm{1}(Z_n \in W)=\operatorname{Freq}_\mu(W)$ $\mathbf{P}_x$-a.s. for every $x\in \Gamma$, which exists by the definition of $\Omega(\Gamma,\mu)$. Note that the set $\Omega(\Gamma,\mu)$ and the function $\operatorname{Freq}_\mu$ are both invariant under 
$\Gamma$.



The following lemma was proven by Lyons and Schramm \cite[Lemma 4.2]{LS99}. (They state their lemma in a slightly different level of generality, but the proof applies to the statement we give here.) Their proof idea can be traced back to Burton and Keane\cite{burton1989density}, who constructed the density for clusters on $\mathbb{Z}^d$.

\begin{lem}[Existence of frequencies]
\label{lem: freq}
Let $\Gamma$ be a countable group and let $\mu$ be a symmetric probability measure on $\Gamma$ whose support generates $\Gamma$. If $\Pi$ is a random equivalence relation on $\Gamma$ whose law is $\Gamma$-invariant, then every component of $\Pi$ a.s.\ has a well-defined $\mu$-frequency in the sense that every component belongs to $\Omega(\Gamma,\mu)$ a.s.
\end{lem}

\begin{corollary}
\label{cor:relative_freq}
Let $G$ be a Cayley graph of a finitely generated group $\Gamma$, let $H$ be a subgroup of $\Gamma$, let $G_p$ be Bernoulli-$p$ bond percolation on $G$ for some $p\in [0,1]$, and let $\mu$ be a symmetric probability measure on $H$ whose support generates $H$. Then $K\cap H$ belongs to $\Omega(H,\mu)$ for every cluster $K$ of $G_p$ almost surely.
\end{corollary}

\begin{proof}
Apply \cref{lem: freq} with $\Gamma=H$ and with $\Pi$ the partition of $H$ into clusters of $G_p$.
\end{proof}

Next we employ the notion of frequency to prove Theorem \ref%
{thm:monotone}. 

\begin{proof}[Proof of Theorem \ref{thm:monotone}]
%
%
If there exists a unique $H$-infinite cluster a.s.\ then we have by the Harris-FKG inequality that
\begin{multline*}\inf_{x,y\in H}\tau_p(x,y)\geq \inf_{x,y\in H}\mathbb{P}_p(x,y \text{ both in $H$-infinite clusters})
\\\geq \inf_{x,y\in H}\mathbb{P}_p(x \text{ in an $H$-infinite cluster})\mathbb{P}_p(y \text{ in an $H$-infinite cluster})=c^2>0,
\end{multline*}
establishing one direction of Theorem \ref{thm:monotone}.

 It remains to prove that if $\inf_{x,y\in H}\tau _{p}(x,y)=c>0$ then there is a unique $H$-infinite cluster almost surely.
Letting $(Z_i)_{i\geq 0}$ be a random walk on $H$ independent from the percolation configuration, with some symmetric step distribution $\mu$ whose support generates $H$, we have that
\[\frac{1}{n} \E \sum_{i=1}^n \mathbbm{1}(Z_i \in K_o) \geq c\]
for every $n\geq 1$ and hence
 by Markov's inequality that 
\begin{equation*}
\mathbb{P}_{p}\left(\frac{1}{n}  \sum_{i=1}^n \mathbbm{1}(Z_i \in K_o) \geq \frac{c}{2}\right)\geq \frac{c}{2}>0
\end{equation*}%
for every $n\geq 1$.
Thus, it follows by Fatou's lemma that the frequency of $K_o$ is at least $c/2$ with positive probability.
Since $\max \{\operatorname{Freq}_\mu(K\cap H):K$ is a cluster of $G_p\}$ is an $%
H$-invariant random variable, it follows that this maximum is equal to some positive contant $c_*$ a.s.\ On the other hand, it follows from the definitions that the frequency $\operatorname{Freq}_\mu$ is finitely additive  and hence that there are at most $1/c_*$ clusters attaining this maximum frequency. If we let $K$ be one of these clusters chosen uniformly at random, then the law of $(G_p,K)$ is $H$-invariant, and the claim follows from \cref{lemma:infinite cluster}.
\end{proof}

\subsection{Proof of \cref{prop:relative pu under wq-normal}}

In this section we first prove co-heredity from $s$-normality (\cref{prop:relative pu}), which states that $G_p$ has a unique $H_1$-infinite cluster iff it has a unique $H_2$-infinite cluster in case $H_1$ is $s$-normal in $H_2$. We then strengthen this result to co-heredity from $wq$-normality (\cref{prop:relative pu under wq-normal}).

\begin{proof}[Proof of Proposition \ref{prop:relative pu}]
First note that if $p$ is such that there exists a unique $H_2$-infinite cluster $K$ a.s., then the law of the pair $(G_p,K)$ is invariant under $H_2$ and hence under $H_1$ also. Thus, it follows from \cref{lemma:infinite cluster} that there is also a unique $H_1$-infinite cluster a.s.\ as desired. Note that this implication  only relies on $H_1$ being a subgroup of $H_2$ and does not require any normality condition.

For the other direction, let $p$ be such that there exists a unique $H_{1}$-infinite
cluster $K$ a.s., so that there exists a constant $c>0$ such that $\P_p(h\in K)=c$ for any $h\in H_1$. For each $\eta \in H_{2}$%
, the graph automorphism $L_\eta$ of $G$ maps $H_1$ to the left coset $\eta H_1$. Thus, for each such left coset there
is also a unique $\eta H_1$-infinite cluster $K_{\eta }$ a.s. and $\P_p(x\in K_\eta)=c$ for any $x\in \eta H_1$. We claim that $K_{\eta }=K$ for every $\eta\in H_2$ a.s.
To this end, let $s_{1},s_{2},\ldots,s_{n}\in S$ be such that $s_{1}s_{2}\cdots s_{n}=\eta $. Recall the assumption that $\left\vert \eta H_{1}\cap H_{1}\eta \right\vert=\left\vert  H_{1}\cap \eta H_{1}\eta^{-1} \right\vert
=\infty $, and note that if $h\in \eta H_1 \eta^{-1}$ then $h\eta \in \eta H_1$.
 For each $h \in \eta H_1 \eta^{-1}\cap H$, define the event 
\begin{equation*}
A_{h}:=\left\{ h\in K,h\eta \in K_{\eta },\text{and for }1\leq j\leq n-1%
\text{ the edge between }hs_{j}\text{ and }hs_{j+1}\text{ is open}%
\right\} .
\end{equation*}%
The events $\{h$ belongs to an $H$-infinite cluster$\}$ and $\{h\eta$ belongs to an $\eta H$-infinite cluster$\}$ are both increasing events, and are equal to the events $\left\{ h\in K\right\} $ and $\left\{ h\eta
\in K_{\eta }\right\} $ respectively up to null sets. Letting
\[E_h:=\left\{ \text{for each }1\leq j\leq n-1\text{ the edge
between }hs_{j}\text{ and }hs_{j+1}\text{ is open}\right\} ,\] which is also an increasing event, we have by the Harris-FKG inequality that
\begin{equation*}
\mathbb{P}_{p}\mathbb{(}A_h\mathbb{)\geq P}_{p}(h\in K)\mathbb{P}%
_{p}(h\eta \in K_{\eta })\mathbb{P}_{p}(E_h)=c^{2}p^{n}
\end{equation*}%
for every $h\in \eta H \eta^{-1}\cap H$.
It follows by Fatou's lemma that
$\mathbb{P}_{p}\mathbb{(}A_{h}$ holds for infinitely many $h\in\eta H\eta^{-1}\cap H) \geq%
\inf_{h\in \eta H \eta^{-1}\cap H} \mathbb{P}_{p}%
\mathbb{(}A_{h}\mathbb{)}\geq c^{2}p^{n}>0$.
Since the event%
\begin{equation*}
\left\{ A_h\text{ holds for infinitely many } h\in \eta H \eta^{-1}\cap H\right\} 
\end{equation*}%
coincides up to null sets with the tail event formed by replacing ``$h\in K$'' and ``$h\eta\in K_\eta$'' with ``$h$ belongs to an $H$-infinite cluster'' and ``$h\eta$ belongs to an $\eta H$-infinite cluster'' in the definition of each event $A_h$, it must in fact have probability $1$. On this event the two clusters $K$ and $K_\eta$ are equal, completing the proof of the claim that the clusters $K$ and $K_\eta$ coincide a.s.\ for every $\eta\in H_2$.
As such, there is a.s.\ a unique cluster $K$ having infinite intersection with every left coset $\eta H_1$ for $\eta \in H_2$. It follows that the law of the pair $(G_p,K)$ is $H_2$-invariant, and hence by \cref{lemma:infinite cluster} that there is a unique $H_2$-infinite cluster a.s.\ as desired.
\end{proof}

The proof of \cref{prop:relative pu under wq-normal} from \cref{prop:relative pu} uses a combination of transfinite induction and \cref{lemma:infinite cluster}.
\begin{proof}[Proof of \cref{prop:relative pu under wq-normal}]
Let $\alpha$ be a countable ordinal number and $\{H_\beta\}_{\beta<\alpha}$ be an ascending $\alpha$-chain of subgroups of $\Gamma$ such that $\cup_{\gamma<\beta} H_\gamma$ is an $s$-normal subgroup of $H_\beta$ for every $\beta\leq \alpha$.
 It suffices to prove that for each $p\in [0,1]$, $G_p$ contains a unique $H_0$-infinite cluster a.s. if and only if it contains a unique $H_\beta$-infinite cluster a.s.\ for every $\beta\leq \alpha$. As noted in the proof of \cref{prop:relative pu}, if a cluster is the unique cluster having infinite intersection with a larger subgroup, then it must be the unique cluster having infinite intersection with a smaller subgroup, so that we need only prove the other direction.

For the other direction, suppose that $K$ is the unique $H_0$-infinite cluster a.s. We shall use transfinite induction to prove that $K$ is also the unique $H_\beta$ infinite cluster for every $\beta \leq \alpha$. If $\beta=\gamma+1$ is a successor ordinal and $K$ is the unique $H_\gamma$-infinite cluster a.s., then \cref{prop:relative pu} immediately gives that $K$ is the unique $H_{\beta}$-infinite cluster a.s. 
The non-trivial case occurs when $\beta$ is a limit ordinal. Suppose that $K$ is the unique $H_\gamma$-infinite cluster a.s for every $\gamma<\beta$. The the law of $K$ is invariant under $H_\gamma$ for every $\gamma<\beta$, and thus invariant under $\cup_{\gamma<\beta} H_\gamma$. Applying \cref{lemma:infinite cluster} to $(G_p,K)$, we deduce that $K$ is the unique $\cup_{\gamma<\beta} H_\gamma$-infinite cluster a.s. Recalling that $\cup_{\gamma<\beta} H_\gamma$ is $s$-normal in $H_\beta$, another application of \cref{prop:relative pu} gives that $K$ is the unique $H_\beta$-infinite cluster a.s. This completes the transfinite induction. 
\end{proof}




    

\section{\label{sec:sharp transition} Relative Sharpness of the Phase Transition}

In this section we prove Theorem \ref{thm:weak sharpness}. To do this, we first prove in \cref{subsec:sharpness tran} a very general theorem, not requiring any symmetry assumptions, stating that if intersections of clusters with some set $A$ are \emph{tight} at some value of $p$, then they have uniform exponential tails at any strictly smaller value of $p$. 
We then prove in Section \ref{subsec:sharpness Cayley} that this tightness threshold coincides with $p_c(H;G)$ when $G$ is a Cayley graph of a group $\Gamma$ and $H$ is a subgroup of $\Gamma$.

\subsection{Sharpness in locally finite graphs}\label{subsec:sharpness tran}

Let $G=(V,E)$ be a locally finite graph (not necessarily transitive) and $A$ be an infinite subset of its vertices $V$. Let \begin{equation*}
\bar{p}_{c}(A;G)=\sup_{p}\Bigl\{\lim_{n}\sup_{x} \mathbb{P}_p(\left\vert K_{x}\cap A\right\vert
>n)=0\Bigr\}
\end{equation*}%
be the threshold for the family of random variables $\{|K_x \cap A|:x\in A\}$ to be tight. The main result in this subsection is the following proposition.

\begin{prop}\label{prop:sharpness}
Let $G=(V,E)$ be a locally finite graph and let $A$ be an infinite subset of $V$, and for each $p\in [0,1]$ and $n\geq 0$ let $Q_p(n):=\sup_{u\in V}\mathbb{P}_p(|K_u\cap A|\geq n)$. For each $p<\bar{p}_{c}(A;G)$ there exists $c_p>0$ such that
\begin{equation*}
Q_p(n)\leq \exp (-c_pn).
\end{equation*}
for every $n\geq 1$. 
\end{prop}

The basic proof idea, borrowed from \cite{1901.10363}, is to use ghost fields and the OSSS inequality to prove a differential inequality describing the family of functions $p\rightarrow Q_p(n)$. Since the differential inequality we obtain is identical to that of \cite{1901.10363}, the deduction of \cref{prop:sharpness} from this differential inequality is exactly the same as well and will not be repeated here. (In fact, the same differential inequality also appeared in the work of Menshikov \cite{MR852458}.) To deal with the fact that $Q_p(n)$ is not obviously continuous in $p$, it is convenient to phrase the main estimate as an \emph{integral inequality} rather than as a differential inequality; this has no bearing on any subsequent analysis\footnote{In our primary case of interest it follows from \cref{lemma:KgH} that $Q_p(n)$ is a continuous function of $p$, which allows us to restate the lemma as an estimate on the lower-right Dini derivatives exactly as in \cite{1901.10363} if so desired; this is not required for the rest of the proof to go through.}. 

\begin{lemma}\label{lm:deineq}
Let $G=(V,E)$ be a locally finite graph and $A$ be a subset of $V$. Then the integral inequality
\begin{equation}\label{eq:Qp}
\log \frac{Q_{p_2}(n)}{Q_{p_1}(n)}\geq 
 2\int_{p_1}^{p_2}\left[ \frac{n(1-e^{-1})}{\sum_{m=1}^{n}Q_p(m)}-1\right]\dif p
\end{equation}%
%
holds for every $0\leq p_1\leq p_2 \leq 1$ and $n\geq 1$.
\end{lemma}

Let us now briefly review the theory of decision trees and the OSSS inequality, which is due to O'Donnell, Saks, Schramm, and Servedio \cite{o2005every}. Let $E$ be a
countable set. A \textbf{decision tree} $T$ consists of $%
e_{1}\in E$ and a sequence of functions 
$S_{n}:(E\times \{0,1\})^{n-1}\rightarrow E$.
Informally, $T$ is a rule for deciding which edge to query at each step based on the status of the edges that have been queried up until that step:
given $\omega \in \{0,1\}^{E}$, the decision tree $T$ first queries the state of $\omega
(e_{1})$. Inductively, after querying the edges $e_1,\ldots,e_{n-1}$ for some $n\geq 2$,
$T$ then chooses to query the edge $e_n$ defined by
\begin{equation*}
e_n=T_{n}(\omega ):=S_{n}(e_{1},\omega (e_{1}),e_{2},\omega
(e_{2}),....,e_{n-1},\omega (e_{n-1})).
\end{equation*}%
 A \textbf{decision forest} $T$
is defined to be a collection of decision trees $F=(T^{i}:i\in I)$ indexed
by a countable set $I$. 
Given a decision tree $T$, we let $\mathcal{F}(T):=\sigma (\{T_{n}(\omega )\})$ be the $\sigma $-algebra
of information revealed by the decision tree, and similarly for a decision forest $F$ we let $\mathcal{F}(F)$ be the smallest $\sigma$-algebra containing all the $\sigma$-algebras $\mathcal{F}(T^i)$.

Let $\mu $ be a probability measure on $\{0,1\}^{E}$ and $\omega $ be a
random variable with law $\mu $. We say a decision forest $F$ \textbf{%
computes} a measurable function $f:\left\{ 0,1\right\} ^{E}\rightarrow
\lbrack -1,1]$ if $f(\omega )$ is  measurable with respect to the $\mu $%
-completion of $\mathcal{F}(F)$. For each $e\in E$, we define the \textbf{revealment
probability} to be%
\begin{equation*}
\delta _{e}(F,\mu ):=\mu (\text{there exists }n\geq 1\text{ and }i\geq 1\text{ such that }%
T_{n}^i(\omega )=e).
\end{equation*}
We will use the following form of the OSSS inequality for decision forests \cite[Corollary 2.4]{1901.10363}.

\begin{lemma}[OSSS inequality for decision forests]
\label{cov}Let $E$ be a finite or countably infinite set and let $\mu $ be a
monotonic measure on $\left\{ 0,1\right\} ^{E}$. Then for every pair of
measurable, $\mu $-integrable functions $f,g:\{0,1\}^{E}\rightarrow \{0,1\}$
with $f$ increasing and every decision forest $F$ computing $g$ we have that%
\begin{equation*}
\left\vert \operatorname{Cov}_{\mu}(f,g)\right\vert \leq \sum_{e\in E}\delta _{e}(F,\mu )\operatorname{Cov}_{\mu
}(f,\omega (e)).
\end{equation*}
\end{lemma}
The relevance of the OSSS inequality to sharp threshold phenomena comes from the following standard lemma  \cite[Proposition 2.1]{1901.10363}, which is an inequality version of Russo's formula for infinite graphs.
This inequality will be stated in terms of the 
\textbf{lower-right Dini derivative}
\[\left( \frac{d}{dx}\right) _{+}f(x):=\liminf_{\varepsilon\downarrow 0}\frac{1}{\varepsilon}(f(x+\varepsilon)-f(x)),\]
which satisfies the inequality version of the fundamental theorem of calculus
\[
f(x_2)-f(x_1) \geq \int_{x_1}^{x_2} \left( \frac{d}{dx}\right) _{+}f(x) \dif x
\]
when $f$ is increasing and $x_1\leq x_2$. 

\begin{lemma}
\label{lem:Russo}
Let $f:\{0,1\}^{E(G)}\rightarrow \mathbb{R}$ be an increasing function. Then
the lower-right Dini derivative of $f$ satisfies%
\begin{equation*}
\left( \frac{d}{dp}\right) _{+}\mathbb{P}_{p}(f(\omega ))\geq \frac{1}{p(1-p)%
}\sum_{e\in E}\operatorname{Cov}_{\mathbb{P}_{p}}(f(\omega ),\omega (e)).
\end{equation*}
\end{lemma}

It follows from this lemma and the OSSS inequality that if there is a decision forest computing $A$ with small maximal revealment, then the logarithmic derivative of $\P_p(A)$ is large.

\begin{proof}[Proof of Lemma \ref{lm:deineq}] 
Recall that $A$ is a subset of vertices in the graph $G$. Independent of $G_{p}$, let $\eta \in
\{0,1\}^{H}$ be a random subset of $A$ where any $v\in A$ is included
independently with probability $1-e^{-1/n}$. The random set $\eta$ is known as a \textbf{ghost field} on $A$. Let $\mathbf{P}_{p,n} $ and $%
\mathbf{E}_{p,n}$ denote probabilities and expectations taken with respect to the joint law of the percolation configuration $G_p$ and the ghost field $\eta$. Let $%
f_{x},g_{x}:\{0,1\}^{E\cup V}$ $\rightarrow \{0,1\}$ be increasing functions
defined by%
\begin{eqnarray*}
f_{x}(\omega ,\eta ) &:&=\mathbbm{1}\{\left\vert K_{x}(\omega )\cap
A\right\vert \geq n\} \\
g_{x}(\omega ,\eta ) &:&=\mathbbm{1}\{\eta (u)=1\text{ for some }u\in
K_{x}(\omega )\cap A\}.
\end{eqnarray*}

For each $u\in A$, we define $T^{u}$ to be a decision tree defined as follows. Start by
querying $\eta (u)$. If $\eta (u)=0$, then the decision tree stops. If $\eta
(u)=1$, at each step after, the decision tree explores an edge whose one endpoint is already
known to connect to $u$, done in such a way that the entire cluster of $u$ is eventually explored. See \cite{1901.10363} for formal details of how to define this tree. This gives that 
\begin{equation*}
\{y\in V\cup E:T_{n}^{u}(\omega ,\eta )=y\text{ for some }n\geq 1\}=\left\{ 
\begin{array}{cc}
\{u\} & \eta (u)=0 \\ 
\{u\}\cup E(K_{u}(\omega )) & \eta (u)=1%
\end{array}%
\right. ,
\end{equation*}%
where $E(K_{u}(\omega ))$ is the set of edges with at least one endpoint in $%
K_{u}(\omega )$. Let $F=\{T^{u}:u\in V\}$ be the resulting decision forest. Then $F$
computes $g_{x}$ and, by the OSSS inequality (Lemma \ref{cov}), 
\begin{eqnarray*}
\operatorname{Cov}[f_{x},g_{x}] &\leq &\sum_{e\in E}\delta
_{e}(F,\mathbf{P}_{p,n})\operatorname{Cov}[f_{x},\omega (e)]+\sum_{u\in
V}\delta _{u}(F,\mathbf{P}_{p,n})\operatorname{Cov}[f_{x},\eta (e)] \\
&=&\sum_{e\in E}\delta _{e}(F,\mathbf{P}_{p,n})\operatorname{Cov}[f_{x},\omega (e)],
\end{eqnarray*}%
where all covariances are taken with respect to the measure $\mathbf{P}_{p,n}$ and the equality on the second line holds because $f$ is independent of $\eta (e)$.

Note that an edge $e$ is revealed by $F(\omega ,\eta )$ if and only if the
cluster at at least one endpoint of $e$ contains a vertex $v$ with $\eta(v)=1$. Writing $\eta
(W)=\sum_{u\in W}\eta (u)$ for each $W\subset V$, it follows that%
\begin{equation*}
\delta _{e}(F,\mathbf{P}_{p,n})\leq 2\sup_{u\in V}\mathbb{P}_{p}(\eta
(K_{u})\geq 1)=2\sup_{u\in V}\mathbb{P}_{p}[1-\exp (- \left\vert
K_{u}\cap A\right\vert /n)]
\end{equation*}%
and hence that
\begin{equation*}
\operatorname{Cov}[f_{x},g_{x}]\leq 2\sup_{u\in V}\mathbb{P}_{p}[1-\exp
(- \left\vert K_{u}\cap A\right\vert /n)]\sum_{e\in E}\operatorname{Cov}[f_{x},\omega (e)].
\end{equation*}%
The left-hand side of the above inequality is 
\begin{align*}
&\mathbf{P}_{p,n}(\left\vert K_{x}\cap A\right\vert \geq n,\eta (K_{x}\cap
A)\geq 1)-\mathbf{P}_{p,n}(\eta (K_{x}\cap A)\geq 1)\mathbb{P}_{p}(\left\vert
K_{x}\cap A\right\vert \geq n) \\
&\hspace{1cm}=\mathbb{E}_{p}\left[ \left( 1-\exp (- \left\vert K_{x}\cap
A\right\vert /n)\right) 1_{\{\left\vert K_{x}\right\vert \geq n\}}\right] -%
\mathbb{E}_{p}\left( 1-\exp (- \left\vert K_{x}\cap A\right\vert
/n)\right) \mathbb{P}_{p}(\left\vert K_{x}\cap A\right\vert \geq n) \\
&\hspace{1cm}\geq (1-e^{-1 })\mathbb{P}_{p}(\left\vert K_{x}\cap A\right\vert
\geq n)-\mathbb{E}_{p}\left( 1-\exp (- \left\vert K_{x}\cap
A\right\vert /n)\right) \mathbb{P}_{p}(\left\vert K_{x}\cap A\right\vert
\geq n).
\end{align*}%
We arrive at%
\begin{eqnarray*}
\sum_{e\in E}\operatorname{Cov}[f_{x},\omega (e)] &\geq &\frac{(1-e^{-1})-\mathbf{P}_{p,n}
\left( 1-\exp (- \left\vert K_{x}\cap A\right\vert /n)\right) }{%
2\sup_{u\in V}\mathbb{E}_{p}[1-\exp (- \left\vert K_{u}\cap
A\right\vert /n)]}\mathbb{P}_{p}(\left\vert K_{x}\cap A\right\vert \geq n) \\
&\geq &\frac{1}{2}\left[ \frac{1-e^{-1}}{\sup_{u\in V}\mathbb{E}%
_{p}[1-\exp (- \left\vert K_{u}\cap A\right\vert /n)]}-1\right] 
\mathbb{P}_{p}(\left\vert K_{x}\cap A\right\vert \geq n).
\end{eqnarray*}%
Using that $p(1-p)\leq 1/4$, we obtain from this and \cref{lem:Russo} that
\begin{equation*}
\left( \frac{d}{dp}\right) _{+}\mathbb{P}_{p}(\left\vert K_{x}\cap A\right\vert \geq n) 
\geq  2\left[ \frac{(1-e^{-1})}{\sup_{u\in V}\mathbb{E}_{p}[1-\exp (- \left\vert K_{u}\cap A\right\vert /n)]}-1\right] 
\mathbb{P}_{p}(\left\vert K_{x}\cap A\right\vert \geq n).  
\end{equation*}%
Using the inequality
\begin{equation*}
\mathbb{P}_{p}[1-\exp (- \left\vert K_{u}\cap A\right\vert /n)]\leq 
\frac{1 }{n}\mathbb{P}_{p}[ n\wedge \left\vert
K_{u}\cap A\right\vert ]=\frac{ 1}{n}\sum_{m=1}^{n }\mathbf{P}_p (\left\vert K_{u}\cap A\right\vert \geq m) \leq \frac{ 1}{n}\sum_{m=1}^{n }Q_p(m)
\end{equation*}%
it follows that
\begin{equation}
\left( \frac{d}{dp}\right) _{+}\mathbb{P}_{p}(\left\vert K_{x}\cap
A\right\vert \geq n)
%
\geq 2\left[ \frac{n(1-e^{-1})}{\sum_{m=1}^{n}Q_p(m)}-1\right] \mathbb{P}_{p}(\left\vert K_{x}\cap
A\right\vert \geq n).
\end{equation}%
Rewriting this as a lower bound on the logarithmic derivative and integrating yields that
\begin{equation}
\log \frac{\mathbb{P}_{p_2}(\left\vert K_{x}\cap
A\right\vert \geq n)}{\mathbb{P}_{p_1}(\left\vert K_{x}\cap
A\right\vert \geq n)}\geq 
2 \int_{p_1}^{p_2}
\left[ \frac{n(1-e^{-1})}{\sum_{m=1}^{n}Q_p(m)}-1\right] \dif p,
\end{equation}
for every $p_1 \leq p_2$, and the claim follows by taking the maximum over $x$.
\end{proof}

\subsection{Sharpness in Cayley graphs}\label{subsec:sharpness Cayley}

In this section, we complete the proof of \cref{thm:weak sharpness} by proving the following lemma, which implies that $\sup_{x} \P_p(|K_x\cap H|\geq n)$ coincides with $\P_p(|K_o \cap H|\geq n)$ up to a factor of $2$ when $G$ is the Cayley graph of a group $\Gamma$ and $H$ is a subgroup of $\Gamma$.

\begin{lemma}\label{lemma:KgH}
Let $H$ be a subgroup of a finitely generated group $\Gamma$, let $S$ be a generating set for $\Gamma$, and let $G=\operatorname{Cay}(\Gamma,S)$. For each $p<p_{c}(H;G)$, we have that
\begin{equation*}
\mathbb{P}_{p}(\left\vert K_{o}\cap \gamma H\right\vert \geq n)+\mathbb{P}_{p}(\vert K_{o}\cap  H \gamma^{-1} \vert \geq n)\leq
2\mathbb{P}_{p}(\left\vert K_{o}\cap H\right\vert \geq n)
\end{equation*}
for every $\gamma \in \Gamma$ and $n\geq 1$. In particular, $\bar{p}_c(H;G)=p_c(H;G)$.
\end{lemma}

\begin{remark}
Intuitively, it is plausible that it should always be easier to have a large intersection with some set if we start inside that set, rather than outside it. This is not always the case, however, even under symmetry assumptions. Indeed, let $T$ be the $d$-regular tree, fix an end of $T$ and orient each edge towards that end. The group of automorphisms of $T$ that preserve this information is transitive but nonunimodular.
This gives a decomposition of $T$ into levels $(L_n)_{n\in \Z}$, where we take the origin to be in level $0$. The set $L_0$ is acted on transitively by a subgroup of $\Aut(T)$, making it at least somewhat analogous to a subgroup.
If $1/(d-1)< p < 1$, then every cluster has finite intersection with $L_0$, but clusters of vertices in level $n$ have intersection with $L_0$ of order $(p(d-1))^n$ with good probability. Thus, in this example $p_c(L_0;T)=1$ but $\bar{p}_c(L_0;T)=1/(d-1)$. Moreover, when $1/(d-1)< p < 1$ it is possible for the cluster of the origin to have a large intersection with $L_0$ by first connecting to a vertex at height $n$, which has probability $p^n$, then connecting from there to a set of points in $L_0$ of order $(p(d-1))^n$. This yields the power-law tail lower bound $\P_p(|K_o \cap L_o|\geq k) \geq c k^{\log p/\log [p(d-1)]}$ for each $1/(d-1)<p<1$, so that \cref{thm:weak sharpness} cannot be extended to this example.
\end{remark}



\begin{proof}[Proof of Lemma \ref{lemma:KgH}]
Fix a group element $\gamma \in \Gamma$. Consider the following function
\[
f(G_p,x,y)=\frac{1}{\vert K_{x}\cap (xH\cup x H \gamma)\vert }%
1\{x\leftrightarrow y,y\in xH\cup x H\gamma,\vert K_{x}\cap x
H\gamma\vert \geq n\}.
\]
where $x\leftrightarrow y$ denotes the event that $x$ is connected to $y$ in $G_p$. Note that $f$ is invariant under  multiplication on the left by elements of $\Gamma$, and therefore satisfies the mass-transport principle
\begin{equation}\label{eqn:msp}
    \mathbb{E}_p\sum_{x} f(G_p,o,x)=\mathbb{E}_p\sum_{x} f(x^{-1}G_p,x^{-1},o)=\mathbb{E}_p\sum_{x} f(G_p,x^{-1},o)=\mathbb{E}_p\sum_{x} f(G_p,x,o).
\end{equation}
The left-hand side of the equality \eqref{eqn:msp} is
\begin{multline*}
\mathbb{E}_p\sum_{y} f(G_p,o,y) =\mathbb{E}_p\frac{1}{\vert K_{o}\cap (H\cup  H\gamma )\vert }%
\sum_{y}\mathbbm{1}\{o\leftrightarrow y,y\in H\cup  H\gamma,\vert K_{o}\cap
 H\gamma \vert \geq n\} \\
=\mathbb{E}_p\frac{1}{\vert K_{o}\cap (H\cup  H \gamma)\vert }\cdot
\vert K_{o}\cap (H\cup  H\gamma)\vert \cdot \mathbbm{1}\{\vert K_{o}\cap
 H \gamma \vert \geq n\} = \mathbb{P}_p(\vert K_{o}\cap  H\gamma\vert \geq n).
\end{multline*}%
Meanwhile, the right-hand side of \eqref{eqn:msp} is 
\begin{align*}
\mathbb{E}_p\sum_{x
} f(G_p,x,o)  &=\mathbb{E}_p\sum_{x}\frac{1}{\vert K_{x}\cap (xH\cup x H\gamma)\vert }%
\mathbbm{1}\{x\leftrightarrow o,o\in xH\cup x H \gamma,\vert K_{x}\cap x
H \gamma\vert \geq n\}  \nonumber \\
&=\mathbb{E}_p\sum_{x}\frac{1}{\vert K_{x}\cap (xH\cup x H \gamma)\vert }%
\mathbbm{1}\{x\leftrightarrow o,o\in xH,\vert K_{x}\cap x H\gamma\vert
\geq n\} \nonumber \\
&\hspace{2cm}+\mathbb{E}_p\sum_{x}\frac{1}{\vert K_{x}\cap (xH\cup x H\gamma)\vert }%
\mathbbm{1}\{x\leftrightarrow o,o\in x H\gamma,\vert K_{x}\cap x
H\gamma\vert \geq n\}.
\end{align*}
In the sum on the second line, the condition that $o\in xH$ is equivalent to $x\in H$, so that $x H\gamma = H\gamma$. Similarly, for the sum on the third line, we write $H^\gamma = \gamma^{-1} H \gamma$ and note that the condition $o\in xH\gamma$ is equivalent to $x\in \gamma^{-1} H = H^\gamma \gamma^{-1}$, so that the above equation can be written
\begin{multline}
\mathbb{E}_p\sum_{x
} f(G_p,x,o)  
=\mathbb{E}_p\sum_{x}\frac{1}{\vert K_{o}\cap (H\cup H \gamma)\vert }%
\mathbbm{1}\{x\leftrightarrow o,x \in H,\vert K_{o}\cap  H\gamma\vert
\geq n\} \\
+\mathbb{E}_p\sum_{x}\frac{1}{\vert K_{o}\cap (H^\gamma \gamma^{-1} \cup H^\gamma )\vert }%
\mathbbm{1}\{x\leftrightarrow o, x\in H^\gamma \gamma^{-1},\vert K_{o}\cap 
H^\gamma \vert \geq n\}. \label{eq:RHS_nice}
\end{multline}
From (\ref{eqn:msp}), it follows that 
\begin{multline}
    \mathbb{P}_p(\vert K_{o}\cap  H\gamma\vert \geq n)=\mathbb{E}_p\sum_{x}\frac{1}{\vert K_{o}\cap (H\cup H \gamma)\vert }%
\mathbbm{1}\{x\leftrightarrow o,x \in H,\vert K_{o}\cap  H\gamma\vert
\geq n\}\\\label{eqn:6}
    +
    \mathbb{E}_p\sum_{x}\frac{1}{\vert K_{o}\cap (H^\gamma \gamma^{-1} \cup H^\gamma )\vert }%
\mathbbm{1}\{x\leftrightarrow o, x\in H^\gamma \gamma^{-1},\vert K_{o}\cap 
H^\gamma \vert \geq n\}.
\end{multline}
If we replace $H$ by the conjugated subgroup $H^\gamma$ and $\gamma$ by $\gamma^{-1}$, we get
\begin{multline*}
\mathbb{P}_p(|K_o \cap H^\gamma \gamma^{-1}| \geq n) =\mathbb{E}_p\sum_{x}\frac{1}{\vert K_{o}\cap (H^\gamma \cup H^\gamma \gamma^{-1})\vert }%
\mathbbm{1}\{x\leftrightarrow o,x \in H^\gamma,\vert K_{o}\cap  H^\gamma \gamma^{-1}\vert
\geq n\} \\
+\mathbb{E}_p\sum_{x}\frac{1}{\vert K_{o}\cap ((H^\gamma)^{\gamma^{-1}} \gamma \cup (H^\gamma)^{\gamma^{-1}} )\vert }%
\mathbbm{1}\{x\leftrightarrow o, x\in (H^\gamma)^{\gamma^{-1}} \gamma,\vert K_{o}\cap 
(H^\gamma)^{\gamma^{-1}} \vert \geq n\}
\end{multline*}
which simplifies to
\begin{multline}\label{eqn:7}
\mathbb{P}_p(|K_o \cap H^\gamma \gamma^{-1}| \geq n) =\mathbb{E}_p\sum_{x}\frac{1}{\vert K_{o}\cap (H^\gamma \gamma^{-1}\cup H^\gamma)\vert }%
\mathbbm{1}\{x\leftrightarrow o,x \in H^\gamma,\vert K_{o}\cap  H^\gamma \gamma^{-1}\vert
\geq n\} \\
+\mathbb{E}_p\sum_{x}\frac{1}{\vert K_{o}\cap (H \cup H \gamma  )\vert }%
\mathbbm{1}\{x\leftrightarrow o, x\in H \gamma,\vert K_{o}\cap 
H \vert \geq n\}. 
\end{multline}
If we add these two equalities (\ref{eqn:6}) and (\ref{eqn:7}) together we get that
\begin{align}
&\mathbb{P}_p(|K_o \cap H \gamma | \geq n)+ \mathbb{P}_p(|K_o \cap H^\gamma \gamma^{-1}| \geq n) =
\nonumber\\ 
&\hspace{1.6cm}\mathbb{E}_p\frac{|K_o \cap H \gamma|\mathbbm{1}(\vert K_{o}\cap 
H \vert \geq n)+|K_o \cap H|\mathbbm{1}(\vert K_{o}\cap 
H \gamma \vert \geq n)}{\vert K_{o}\cap (H \cup H \gamma )\vert }
\nonumber\\
&\hspace{3.2cm}+\mathbb{E}_p\frac{|K_o \cap H^\gamma \gamma^{-1}|\mathbbm{1}(\vert K_{o}\cap 
H^\gamma \vert \geq n)+|K_o \cap H^\gamma|\mathbbm{1}(\vert K_{o}\cap 
H^\gamma \gamma^{-1} \vert \geq n)}{\vert K_{o}\cap (H^\gamma \gamma^{-1} \cup H^\gamma )\vert }.
\label{eq:RHS_final}
\end{align}
On the other hand, performing an exactly analogous calculation with the function
\[g(G_p,x,y):=\frac{1}{|K_x\cap (xH\cup x\gamma H)|}\mathbbm{1}\{x\leftrightarrow y,y\in xH\cup x\gamma H,|K_x\cap xH|\geq n\}\]
yields that
\begin{multline}
2\mathbb{P}_p(|K_o \cap H| \geq n) =
\mathbb{E}_p\frac{|K_o \cap H |\mathbbm{1}(\vert K_{o}\cap 
H \vert \geq n)+|K_o \cap H\gamma|\mathbbm{1}(\vert K_{o}\cap 
H \gamma \vert \geq n)}{\vert K_{o}\cap (H \cup H \gamma)\vert }\\
\mathbb{E}_p\frac{|K_o \cap H^\gamma \gamma^{-1}|\mathbbm{1}(|K_o \cap H^\gamma \gamma^{-1}|\geq n)+|K_o \cap H^\gamma|\mathbbm{1}(|K_o \cap H^\gamma|\geq n)}{\vert K_{o}\cap (H^\gamma \gamma^{-1} \cup H^\gamma )\vert }.
\label{eq:RHS_final2}
\end{multline}
It follows from the rearrangement inequality (i.e., the statement that if $x$ and $y$ are non-negative reals and $f$ is a non-negative increasing function then $xf(y)+yf(x) \leq xf(x)+yf(y)$)
that the right hand side of \eqref{eq:RHS_final2} is at least the right hand side of \eqref{eq:RHS_final}.
\end{proof}

\section{Amenable groups of exponential growth are obstacles to uniqueness}\label{sec:obstacles}

We now prove \cref{thm:amenable}; the proof is a relativized version of the proof of \cite{Hutchcroft2016944}.

\begin{proof}[Proof of Theorem \ref{thm:amenable}]
 Let $G$ be a Cayley graph of a finitely generated group $\tilde{\Gamma}$ containing $\Gamma$ as a subgroup, and let $H$ and $S \subseteq \Gamma$ be as in the statement of the theorem. Since $H$ is an amenable $wq$-normal subgroup of $\Gamma$, Proposition \ref{prop:relative pu under wq-normal} and Theorem \ref{thm:BK} give that $p_c(H;G)=p_u(H;G)=p_u(\Gamma;G)$.
Let $d_S$ denote the word metric on the group $\langle S\rangle$, let $B_n(o)$ be the ball of radius $n$ around the origin in $S$ under this metric, and let
\begin{equation*}
\kappa _{p}(n):=\inf \left\{ \tau_p(x,y):x,y\in
 \langle S \rangle,d_S(x,y)\leq n\right\} .
\end{equation*}%
It follows from the same argument as in \cite[Lemma 4]{Hutchcroft2016944} that $\kappa _{p}(n)$ is supermultiplicative, so that $\sup_n \kappa_p(n)^{1/n}=\lim_n \kappa _{p}(n)^{1/n}$ by Fekete's lemma. We also have the trivial inequality
\begin{equation*}
\kappa _{p}(n)\cdot \left\vert B_n(o)\cap H\right\vert \leq \sum_{x\in B_n(o) \cap H} \tau_p(0,x)\leq \mathbb{E}_p%
\left\vert K_{o}\cap H\right\vert
\end{equation*}%
for every $p\in [0,1]$ and $n\geq 0$.
When $p<p_c(H;G)=p_u(H,G)=p_u(\Gamma,G)$, we have by Theorem \ref{thm:weak sharpness} that
$\mathbb{E}_{p}\left\vert K_{o}\cap H\right\vert <\infty$, and hence that
\begin{equation*}
\sup_n \kappa _{p}(n)^{1/n} = \lim_n \kappa _{p}(n)^{1/n}\leq \limsup_n \left( \frac{\mathbb{E}\left\vert
K_{0}\cap H\right\vert }{\left\vert B_n(o)\cap H\right\vert }\right)
^{1/n}= \limsup_n \left( \frac{1}{\left\vert B_n(o)\cap H\right\vert }%
\right) ^{1/n}.
\end{equation*}%
Since $\left\vert B_n(o)\cap H\right\vert $ grows exponentially in $n$, the right hand side is a constant strictly less than $1$ that does not depend on the choice of $p<p_u(\Gamma;G)$. Denoting this constant by $\operatorname{gr}(H;S)^{-1}$,
 we have by continuity of $\kappa _{p}(n)$ that the same estimate $\kappa_p(n)\leq \operatorname{gr}(H;S)^{-n}$ holds for all $p\leq p_u$. This implies that $\inf_{x,y\in \Gamma}\tau_{p_{u}(\Gamma;G)}(x,y)=0$ and hence that there does not exist an unique $\Gamma$-infinite cluster at $p_{u}(\Gamma;G)$ as desired. 
\end{proof}

\section{A problem of Lyons and Schramm}\label{sec:random walks}



In this section, we explain how our relative Burton-Keane theorem for hyperfinite locally unimodular random rooted graphs can also be used to settle a problem posed by Lyons and Schramm in their landmark work \cite{LS99}.
%
%
%
In this work, 
Lyons and Schramm \cite{LS99} used their indistinguishability theorem to prove that if Bernoulli percolation (or any other inserion-tolerant, automorphism-invariant percolation process) on a Cayley graph $G$ has infinitely many infinite clusters, then every such cluster has \emph{zero frequency} in the sense that 
\[\operatorname{Freq}_\mu(K)=\lim_n \frac{1}{n}\sum_{k=1}^n \mathbbm{1}(Z_k\in K)=0,\]
 $\mu$-a.s., where $\mu$ is the law of a simple random walk $(Z_k)_{k\geq 1}$ on $G$ independent of $G_p$. (See Section \ref{subsec: H-freq} for more discussions on frequencies).
%
%
They asked \cite[Question 4.3]{LS99} whether this conclusion can be strengthened to the statement that
\[
\tau_p(Z_0,Z_k)\rightarrow 0 \qquad \text{ $\mu$-a.s}
\]
 We give a positive answer to this question by proving the following stronger result.

\begin{theorem}\label{thm:rw}
Let $G$ be a connected, locally finite, unimodular transitive graph, let $\omega \in \{0,1\}^E$ be a random subgraph of $G$ whose law is automorphism invariant and insertion tolerant, and let $(Z_n)_{n\geq 0}$ be an independent simple random walk on $G$. If $\omega$ has infinitely many infinite clusters then $Z$ visits each cluster of $\omega$ at most finitely many times almost surely.

\end{theorem}


\begin{proof}[Proof of Theorem \ref{thm:rw}]
The assumption that $\omega$ has infinitely many infinite clusters a.s.\ ensures by Burton--Keane that $G$ is nonamenable and hence transient.
Extend $Z$ to a doubly-infinite random walk $(Z_n)_{n\in \Z}$ by taking $(Z_{-n})_{n\geq 0}$ to be an independent simple random walk started at $o$. It is proven in \cite{hutchcroft2020non} that the tuple $(G,\{Z_n:n\in \Z\},o,\omega)$ is hyperfinite quasi-locally unimodular in the sense that it has law absolutely continuous with respect to that of a hyperfinite locally unimodular measure. As such, it follows from Theorem \ref{cor:unique infinite cluster} (the relative Burton-Keane theorem) that there is at most one $\{Z_n:n\in \Z\}$-infinite cluster of $\omega$ a.s. On the other hand, if there were to exist exactly one such cluster $K$ a.s., then it would have $\lim_{n\to \infty} \frac{1}{n} \mathbbm{1}(Z_n \in K)=\P(o$ in a $\{Z_n:n\in \Z\}$-infinite cluster of $\omega)>0$ a.s.\ by the ergodic theorem, which is not possible since every cluster of $\omega$ has zero frequency as proven in \cite{LS99}. \qedhere
\end{proof}

\subsection*{Acknowledgements} This work was supported by NSF grant DMS-2246494. TH thanks Nicol\'as Matte Bon for pointing out to us several examples of amenable $s$-normal subgroups of exponential growth that are not normal during helpful discussions at the 16th annual conference on  Geometric and Asymptotic Group Theory with Applications (GAGTA) at the Erwin Schrödinger International Institute for Mathematics and Physics (ESI) in Vienna in 2023.

\footnotesize{
\bibliographystyle{abbrv}
\bibliography{big_bib_file}
}

\end{document}